\documentclass[11pt,a4paper]{article}
\usepackage{indentfirst}
\usepackage{amsfonts}
\usepackage{amssymb}
\usepackage{mathrsfs}
\usepackage{amsmath}
\usepackage{amsthm}
\usepackage{enumerate}
\usepackage{cite}
\usepackage{mathrsfs}
\allowdisplaybreaks
\usepackage{geometry}
\usepackage{ifpdf}
\ifpdf
\usepackage[colorlinks=true,linkcolor=blue,citecolor=red, final,backref=page,hyperindex]{hyperref}
\else
\usepackage[colorlinks,final,backref=page,hyperindex,hypertex]{hyperref}
\fi

\usepackage{dsfont}
\usepackage{indentfirst}
\usepackage{latexsym}
\usepackage[all]{xy}

\newcommand {\emptycomment}[1]{} 

\newcommand{\nc}{\newcommand}
\newcommand{\delete}[1]{}
\nc{\mfootnote}[1]{\footnote{#1}} 
\nc{\todo}[1]{\tred{To do:} #1}

\delete{
\nc{\mlabel}[1]{\label{#1}}  
\nc{\mcite}[1]{\cite{#1}}  
\nc{\mref}[1]{\ref{#1}}  
\nc{\meqref}[1]{\ref{#1}} 
\nc{\mbibitem}[1]{\bibitem{#1}} 
}

\nc{\mlabel}[1]{\label{#1}  
{\hfill \hspace{1cm}{\bf{{\ }\hfill(#1)}}}}
\nc{\mcite}[1]{\cite{#1}{{\bf{{\ }(#1)}}}}  
\nc{\mref}[1]{\ref{#1}{{\bf{{\ }(#1)}}}}  
\nc{\meqref}[1]{\eqref{#1}{{\bf{{\ }(#1)}}}} 
\nc{\mbibitem}[1]{\bibitem[\bf #1]{#1}} 

\newtheorem{thm}{Theorem}[section]
\newtheorem{lem}[thm]{Lemma}
\newtheorem{cor}[thm]{Corollary}
\newtheorem{pro}[thm]{Proposition}
\theoremstyle{definition}
\newtheorem{defi}[thm]{Definition}
\newtheorem{ex}[thm]{Example}
\newtheorem{rmk}[thm]{Remark}
\newtheorem{pdef}[thm]{Definition-Proposition}

\nc{\tred}[1]{\textcolor{red}{#1}}
\nc{\tblue}[1]{\textcolor{blue}{#1}}
\nc{\tgreen}[1]{\textcolor{green}{#1}}
\nc{\tpurple}[1]{\textcolor{purple}{#1}}
\nc{\btred}[1]{\textcolor{red}{\bf #1}}
\nc{\btblue}[1]{\textcolor{blue}{\bf #1}}
\nc{\btgreen}[1]{\textcolor{green}{\bf #1}}
\nc{\btpurple}[1]{\textcolor{purple}{\bf #1}}

\nc{\jt}[1]{\textcolor{yellow}{JT:#1}}
\nc{\cm}[1]{\textcolor{red}{CM:#1}}
\nc{\liu}[1]{\textcolor{blue}{Liu:#1}}
\nc{\jf}[1]{\textcolor{blue}{#1}}

\nc{\twovec}[2]{\left(\begin{array}{c} #1 \\ #2\end{array} \right )}
\nc{\threevec}[3]{\left(\begin{array}{c} #1 \\ #2 \\ #3 \end{array}\right )}
\nc{\twomatrix}[4]{\left(\begin{array}{cc} #1 & #2\\ #3 & #4 \end{array} \right)}
\nc{\threematrix}[9]{{\left(\begin{matrix} #1 & #2 & #3\\ #4 & #5 & #6 \\ #7 & #8 & #9 \end{matrix} \right)}}
\nc{\twodet}[4]{\left|\begin{array}{cc} #1 & #2\\ #3 & #4 \end{array} \right|}

\nc{\rk}{\mathrm{r}}

\newcommand{\frkl}{\rho}
\newcommand{\frkr}{\mu}

\nc{\tforall}{\text{ for all }}

\nc{\svec}[2]{{\tiny\left(\begin{matrix}#1\\
#2\end{matrix}\right)\,}}  
\nc{\ssvec}[2]{{\tiny\left(\begin{matrix}#1\\
#2\end{matrix}\right)\,}} 

\nc{\typeI}{local cocycle $3$-Lie bialgebra\xspace}
\nc{\typeIs}{local cocycle $3$-Lie bialgebras\xspace}
\nc{\typeII}{double construction $3$-Lie bialgebra\xspace}
\nc{\typeIIs}{double construction $3$-Lie bialgebras\xspace}

\nc{\bia}{{$\mathcal{P}$-bimodule ${\bf k}$-algebra}\xspace}
\nc{\bias}{{$\mathcal{P}$-bimodule ${\bf k}$-algebras}\xspace}

\nc{\rmi}{{\mathrm{I}}}
\nc{\rmii}{{\mathrm{II}}}
\nc{\rmiii}{{\mathrm{III}}}
\nc{\pr}{{\mathrm{pr}}}

\nc{\OT}{constant $\theta$-}
\nc{\T}{$\theta$-}
\nc{\IT}{inverse $\theta$-}

\nc{\pll}{\beta}
\nc{\plc}{\epsilon}

\nc{\ass}{{\mathit{Ass}}}
\nc{\lie}{{\mathit{Lie}}}
\nc{\comm}{{\mathit{Comm}}}
\nc{\dend}{{\mathit{Dend}}}
\nc{\zinb}{{\mathit{Zinb}}}
\nc{\tdend}{{\mathit{TDend}}}
\nc{\prelie}{{\mathit{preLie}}}
\nc{\postlie}{{\mathit{PostLie}}}
\nc{\quado}{{\mathit{Quad}}}
\nc{\octo}{{\mathit{Octo}}}
\nc{\ldend}{{\mathit{ldend}}}
\nc{\lquad}{{\mathit{LQuad}}}

 \nc{\adec}{\check{;}} \nc{\aop}{\alpha}
\nc{\dftimes}{\widetilde{\otimes}} \nc{\dfl}{\succ} \nc{\dfr}{\prec}
\nc{\dfc}{\circ} \nc{\dfb}{\bullet} \nc{\dft}{\star}
\nc{\dfcf}{{\mathbf k}} \nc{\apr}{\ast} \nc{\spr}{\cdot}
\nc{\twopr}{\circ} \nc{\tspr}{\star} \nc{\sempr}{\ast}
\nc{\disp}[1]{\displaystyle{#1}}
\nc{\bin}[2]{ (_{\stackrel{\scs{#1}}{\scs{#2}}})}  
\nc{\binc}[2]{ \left (\!\! \begin{array}{c} \scs{#1}\\
    \scs{#2} \end{array}\!\! \right )}  
\nc{\bincc}[2]{  \left ( {\scs{#1} \atop
    \vspace{-.5cm}\scs{#2}} \right )}  
\nc{\sarray}[2]{\begin{array}{c}#1 \vspace{.1cm}\\ \hline
    \vspace{-.35cm} \\ #2 \end{array}}
\nc{\bs}{\bar{S}} \nc{\dcup}{\stackrel{\bullet}{\cup}}
\nc{\dbigcup}{\stackrel{\bullet}{\bigcup}} \nc{\etree}{\big |}
\nc{\la}{\longrightarrow} \nc{\fe}{\'{e}} \nc{\rar}{\rightarrow}
\nc{\dar}{\downarrow} \nc{\dap}[1]{\downarrow
\rlap{$\scriptstyle{#1}$}} \nc{\uap}[1]{\uparrow
\rlap{$\scriptstyle{#1}$}} \nc{\defeq}{\stackrel{\rm def}{=}}
\nc{\dis}[1]{\displaystyle{#1}} \nc{\dotcup}{\,
\displaystyle{\bigcup^\bullet}\ } \nc{\sdotcup}{\tiny{
\displaystyle{\bigcup^\bullet}\ }} \nc{\hcm}{\ \hat{,}\ }
\nc{\hcirc}{\hat{\circ}} \nc{\hts}{\hat{\shpr}}
\nc{\lts}{\stackrel{\leftarrow}{\shpr}}
\nc{\rts}{\stackrel{\rightarrow}{\shpr}} \nc{\lleft}{[}
\nc{\lright}{]} \nc{\uni}[1]{\tilde{#1}} \nc{\wor}[1]{\check{#1}}
\nc{\free}[1]{\bar{#1}} \nc{\den}[1]{\check{#1}} \nc{\lrpa}{\wr}
\nc{\curlyl}{\left \{ \begin{array}{c} {} \\ {} \end{array}
    \right .  \!\!\!\!\!\!\!}
\nc{\curlyr}{ \!\!\!\!\!\!\!
    \left . \begin{array}{c} {} \\ {} \end{array}
    \right \} }
\nc{\leaf}{\ell}       
\nc{\longmid}{\left | \begin{array}{c} {} \\ {} \end{array}
    \right . \!\!\!\!\!\!\!}
\nc{\ot}{\otimes} \nc{\sot}{{\scriptstyle{\ot}}}
\nc{\otm}{\overline{\ot}}
\nc{\ora}[1]{\stackrel{#1}{\rar}}
\nc{\ola}[1]{\stackrel{#1}{\la}}
\nc{\pltree}{\calt^\pl}
\nc{\epltree}{\calt^{\pl,\NC}}
\nc{\rbpltree}{\calt^r}
\nc{\scs}[1]{\scriptstyle{#1}} \nc{\mrm}[1]{{\rm #1}}
\nc{\dirlim}{\displaystyle{\lim_{\longrightarrow}}\,}
\nc{\invlim}{\displaystyle{\lim_{\longleftarrow}}\,}
\nc{\mvp}{\vspace{0.5cm}} \nc{\svp}{\vspace{2cm}}
\nc{\vp}{\vspace{8cm}} \nc{\proofbegin}{\noindent{\bf Proof: }}
\nc{\proofend}{$\blacksquare$ \vspace{0.5cm}}
\nc{\freerbpl}{{F^{\mathrm RBPL}}}
\nc{\sha}{{\mbox{\cyr X}}}  
\nc{\ncsha}{{\mbox{\cyr X}^{\mathrm NC}}} \nc{\ncshao}{{\mbox{\cyr
X}^{\mathrm NC,\,0}}}
\nc{\shpr}{\diamond}    
\nc{\shprm}{\overline{\diamond}}    
\nc{\shpro}{\diamond^0}    
\nc{\shprr}{\diamond^r}     
\nc{\shpra}{\overline{\diamond}^r}
\nc{\shpru}{\check{\diamond}} \nc{\catpr}{\diamond_l}
\nc{\rcatpr}{\diamond_r} \nc{\lapr}{\diamond_a}
\nc{\sqcupm}{\ot}
\nc{\lepr}{\diamond_e} \nc{\vep}{\varepsilon} \nc{\labs}{\mid\!}
\nc{\rabs}{\!\mid} \nc{\hsha}{\widehat{\sha}}
\nc{\lsha}{\stackrel{\leftarrow}{\sha}}
\nc{\rsha}{\stackrel{\rightarrow}{\sha}} \nc{\lc}{\lfloor}
\nc{\rc}{\rfloor}
\nc{\tpr}{\sqcup}
\nc{\nctpr}{\vee}
\nc{\plpr}{\star}
\nc{\rbplpr}{\bar{\plpr}}
\nc{\sqmon}[1]{\langle #1\rangle}
\nc{\forest}{\calf}
\nc{\altx}{\Lambda_X} \nc{\vecT}{\vec{T}} \nc{\onetree}{\bullet}
\nc{\Ao}{\check{A}}
\nc{\seta}{\underline{\Ao}}
\nc{\deltaa}{\overline{\delta}}
\nc{\trho}{\tilde{\rho}}

\nc{\rpr}{\circ}
\nc{\dpr}{{\tiny\diamond}}
\nc{\rprpm}{{\rpr}}

\nc{\mmbox}[1]{\mbox{\ #1\ }} \nc{\ann}{\mrm{ann}}
\nc{\Aut}{\mrm{Aut}} \nc{\can}{\mrm{can}}
\nc{\twoalg}{{two-sided algebra}\xspace}
\nc{\colim}{\mrm{colim}}
\nc{\Cont}{\mrm{Cont}} \nc{\rchar}{\mrm{char}}
\nc{\cok}{\mrm{coker}} \nc{\dtf}{{R-{\rm tf}}} \nc{\dtor}{{R-{\rm
tor}}}

\nc{\depth}{{\mrm d}}
\nc{\Div}{{\mrm Div}} \nc{\End}{\mrm{End}} \nc{\Ext}{\mrm{Ext}}
\nc{\Fil}{\mrm{Fil}} \nc{\Frob}{\mrm{Frob}} \nc{\Gal}{\mrm{Gal}}
\nc{\GL}{\mrm{GL}} \nc{\Hom}{\mrm{Hom}} \nc{\hsr}{\mrm{H}}
\nc{\hpol}{\mrm{HP}} \nc{\id}{\mrm{id}} \nc{\im}{\mrm{im}}
\nc{\incl}{\mrm{incl}} \nc{\length}{\mrm{length}}
\nc{\LR}{\mrm{LR}} \nc{\mchar}{\rm char} \nc{\NC}{\mrm{NC}}
\nc{\mpart}{\mrm{part}} \nc{\pl}{\mrm{PL}}
\nc{\ql}{{\QQ_\ell}} \nc{\qp}{{\QQ_p}}
\nc{\rank}{\mrm{rank}} \nc{\rba}{\rm{RBA }} \nc{\rbas}{\rm{RBAs }}
\nc{\rbpl}{\mrm{RBPL}}
\nc{\rbw}{\rm{RBW }} \nc{\rbws}{\rm{RBWs }} \nc{\rcot}{\mrm{cot}}
\nc{\rest}{\rm{controlled}\xspace}
\nc{\rdef}{\mrm{def}} \nc{\rdiv}{{\rm div}} \nc{\rtf}{{\rm tf}}
\nc{\rtor}{{\rm tor}} \nc{\res}{\mrm{res}} \nc{\SL}{\mrm{SL}}
\nc{\Spec}{\mrm{Spec}} \nc{\tor}{\mrm{tor}} \nc{\Tr}{\mrm{Tr}}
\nc{\mtr}{\mrm{sk}}

\nc{\ab}{\mathbf{Ab}} \nc{\Alg}{\mathbf{Alg}}
\nc{\Algo}{\mathbf{Alg}^0} \nc{\Bax}{\mathbf{Bax}}
\nc{\Baxo}{\mathbf{Bax}^0} \nc{\RB}{\mathbf{RB}}
\nc{\RBo}{\mathbf{RB}^0} \nc{\BRB}{\mathbf{RB}}
\nc{\Dend}{\mathbf{DD}} \nc{\bfk}{{\bf k}} \nc{\bfone}{{\bf 1}}
\nc{\base}[1]{{a_{#1}}} \nc{\detail}{\marginpar{\bf More detail}
    \noindent{\bf Need more detail!}
    \svp}
\nc{\Diff}{\mathbf{Diff}} \nc{\gap}{\marginpar{\bf
Incomplete}\noindent{\bf Incomplete!!}
    \svp}
\nc{\FMod}{\mathbf{FMod}} \nc{\mset}{\mathbf{MSet}}
\nc{\rb}{\mathrm{RB}} \nc{\Int}{\mathbf{Int}}
\nc{\Mon}{\mathbf{Mon}}
\nc{\remarks}{\noindent{\bf Remarks: }}
\nc{\OS}{\mathbf{OS}} 
\nc{\Rep}{\mathbf{Rep}}
\nc{\Rings}{\mathbf{Rings}} \nc{\Sets}{\mathbf{Sets}}
\nc{\DT}{\mathbf{DT}}

\nc{\BA}{{\mathbb A}} \nc{\CC}{{\mathbb C}} \nc{\DD}{{\mathbb D}}
\nc{\EE}{{\mathbb E}} \nc{\FF}{{\mathbb F}} \nc{\GG}{{\mathbb G}}
\nc{\HH}{{\mathbb H}} \nc{\LL}{{\mathbb L}} \nc{\NN}{{\mathbb N}}
\nc{\QQ}{{\mathbb Q}} \nc{\RR}{{\mathbb R}} \nc{\BS}{{\mathbb{S}}} \nc{\TT}{{\mathbb T}}
\nc{\VV}{{\mathbb V}} \nc{\ZZ}{{\mathbb Z}}

\nc{\calao}{{\mathcal A}} \nc{\cala}{{\mathcal A}}
\nc{\calc}{{\mathcal C}} \nc{\cald}{{\mathcal D}}
\nc{\cale}{{\mathcal E}} \nc{\calf}{{\mathcal F}}
\nc{\calfr}{{{\mathcal F}^{\,r}}} \nc{\calfo}{{\mathcal F}^0}
\nc{\calfro}{{\mathcal F}^{\,r,0}} \nc{\oF}{\overline{F}}
\nc{\calg}{{\mathcal G}} \nc{\calh}{{\mathcal H}}
\nc{\cali}{{\mathcal I}} \nc{\calj}{{\mathcal J}}
\nc{\call}{{\mathcal L}} \nc{\calm}{{\mathcal M}}
\nc{\caln}{{\mathcal N}} \nc{\calo}{{\mathcal O}}
\nc{\calp}{{\mathcal P}} \nc{\calq}{{\mathcal Q}} \nc{\calr}{{\mathcal R}}
\nc{\calt}{{\mathcal T}} \nc{\caltr}{{\mathcal T}^{\,r}}
\nc{\calu}{{\mathcal U}} \nc{\calv}{{\mathcal V}}
\nc{\calw}{{\mathcal W}} \nc{\calx}{{\mathcal X}}
\nc{\CA}{\mathcal{A}}

\nc{\fraka}{{\mathfrak a}} \nc{\frakB}{{\mathfrak B}}
\nc{\frakb}{{\mathfrak b}} \nc{\frakd}{{\mathfrak d}}
\nc{\oD}{\overline{D}}
\nc{\frakF}{{\mathfrak F}} \nc{\frakg}{{\mathfrak g}}
\nc{\frakm}{{\mathfrak m}} \nc{\frakM}{{\mathfrak M}}
\nc{\frakMo}{{\mathfrak M}^0} \nc{\frakp}{{\mathfrak p}}
\nc{\frakS}{{\mathfrak S}} \nc{\frakSo}{{\mathfrak S}^0}
\nc{\fraks}{{\mathfrak s}} \nc{\os}{\overline{\fraks}}
\nc{\frakT}{{\mathfrak T}}
\nc{\oT}{\overline{T}}
\nc{\frakX}{{\mathfrak X}} \nc{\frakXo}{{\mathfrak X}^0}
\nc{\frakx}{{\mathbf x}}
\nc{\frakTx}{\frakT}      
\nc{\frakTa}{\frakT^a}        
\nc{\frakTxo}{\frakTx^0}   
\nc{\caltao}{\calt^{a,0}}   
\nc{\ox}{\overline{\frakx}} \nc{\fraky}{{\mathfrak y}}
\nc{\frakz}{{\mathfrak z}} \nc{\oX}{\overline{X}}

\font\cyr=wncyr10

\nc{\al}{\alpha}
\nc{\lam}{\lambda}
\nc{\lr}{\longrightarrow}

\def\c{\star}
\def\la{\langle}

\allowdisplaybreaks
\begin{document}

\title{ \sf Bialgebra theory and $\mathcal O$-operators of admissible Hom-Poisson algebras}
\author{Karima Benali}
\author{\bf  Karima Benali\footnote{karimabenali172@yahoo.fr}\\\\
{ University of Sfax, Faculty of Sciences,  BP
1171, 3038 Sfax, Tunisia. }}

\date{}	
	
	\maketitle

\begin{abstract}
In this paper, we present and explore several key concepts within the framework of Hom-Poisson algebras. Specifically, we introduce the notions of admissible Hom-Poisson algebras, along with the related ideas of matched pairs and Manin triples for such algebras. We then define the concept of a purely admissible Hom-Poisson bialgebra, placing particular emphasis on its compatibility with the Manin triple structure associated with a nondegenerate symmetric bilinear form. This compatibility is crucial for understanding the structural interplay between these algebraic objects.
Additionally, we investigate the notion of Hom-$
\mathcal O$-operators acting on admissible Hom-Poisson algebras. We analyze their properties and establish a connection with admissible Hom-pre-Poisson algebras, shedding light on the relationship between these two structures.

\end{abstract}

{\bf Keywords}:{ Hom-Poisson algebra, addmissible Hom-Poisson algebra,  representation, addmissible Hom-Poisson bialgebra, Hom-$\mathcal{O}$-operator.}

{\bf MSC (2020)}: 17B61, 16T10, 16T25, 17B63, 	17B62.

\tableofcontents

\numberwithin{equation}{section}

\section{Introduction} Hom-algebra structure is a topic that has been extensively studied in recent years.
Hom-Lie algebras were introduced by Hartwig, Larsson and Silvestrov when studying deformations of the Witt and Viraosoro algebra \cite{HLS,MS}. They are caracterized by an $\alpha$-twisted Jacobi identity, where $\alpha$ is a linear map, called a Hom-Jacobi identity. Lie algebras are a subclass of Hom-Lie algebra when $\alpha$ is the identity map. Some q-deformations of the Witt and the Virasoro algebras have the structure of
a Hom-Lie algebra \cite{HLS} thanks to their close relationship with discrete and deformed vector fields and
differential calculus \cite{HLS,LS1, LS2}.  Hom-Lie algebras are studied in various contexts; representation
and cohomology theory \cite{SM,MS,DY} deformation theory \cite{MS}, categorification theory \cite{SZ} bialgebra theory \cite{BY,CW,CS,MSA}. See also \cite{MS,SZ} for other sudies done on Hom-type algebras.

The notion of Lie bialgebra dates back to the early $80$'s, it was introduced by V.G.Drinfeld \cite{Dr}
in the study of the algebraic realization of Poisson-Lie group. 
 Let $\mathfrak{g}$ be a Lie algebra.
A bialgebra structure is obtained by a corresponding set of comultiplications together with a set of compatibility conditions between the multiplications and the comultiplications.
 In the Hom-setting, there are two approaches to study Hom-Lie bialgebras; The one given in \cite{DY1} where the compatibility condition stands for the cobrackt $\Delta$ is a $1$-cocycle on the Hom-Lie algebra $\mathfrak{g}$. An other one says that  there is not a natural Hom-Lie algebra structure
on $\mathfrak{g}\oplus \mathfrak{g}^*$ such that $(\mathfrak{g}\oplus \mathfrak{g}^*,\mathfrak{g}, \mathfrak{g}^*)$ is a Manin triple for Hom-Lie algebras. To solve this problem, Sheng and Cai introduced the notion of Hom-Lie bialgebras, see for instance \cite{CS}. These structures, called purely Hom-Lie bialgebras, are defined on regular Hom-Lie algebras, namely Hom-Lie algebras with invertible $\alpha$. As a result, many studies have concentrated on bialgebras theory of diverse algebraic structures, including Leibniz algebra, mock-Lie algebras Hom-Lie superalgebras, compatible mock-Lie algebras, anti-Leibniz algebras, $n$-ary algebras, etc (see \cite{Karima, KTAS, MSA, FSS, SB, Hou}).

Poisson algebra named from  Siméon Poisson is both a Lie algebra and a commutative associative algebra combined in certain sense. That is
a Poisson algebra is a triple $(A,[\cdot,\cdot],\circ)$ where $(A,[\cdot,\cdot])$ is a Lie algebra and $(A,\circ)$ is a commutative associative algebra satisfying a compatibilty condition. These algebras play important role in geometry \cite{Wei77,Vaisman1}, classical and quantum mechanics \cite{Arn78,Dirac64,OdA},  algebraic geometry \cite{GK04,Pol97}, quantization theory \cite{Hue90,Kon03,Rem21} and quantum groups \cite{CP1,Dr87}, see also \cite{NB1} for more details on Poisson bialgebras.
The admissible Poisson structure was given in \cite{MR06}, it was studied via
the polarization-depolarization process which was given in \cite {LiLo}, see also \cite{MR06};
 that is, polarization deals with structure with one operation as structure with one commutative and anti-commutative operations, whereas depolarization interprets structure with one commutative and anti-commutative operation as structure with one operation.
Given a Poisson algebra, we can have an algebra with one operation via the depolarization concept, the resulted structure is called admissible Poisson algebra denoted $(A,\star)$. Conversely, given an admissible Poisson algebra we can obtain the corresponding Poisson algebra through the polarization concept.
These algebras were explicitely studied in \cite{Bai3} and since then a lot of works are done on it like admissible Poisson superalgebra \cite{Rem12}.

Admissible Poisson bialgebras theory was established in\cite{Bai3} where the authors characterize these structure by means of matched pairs and Manin triple. 
In this paper, we give the Hom-analogues of these structures by investigating their properties.
We introduce the notion of admissible Hom-Poisson bialgebra as an equivalent structure of a Manin triple of admissible Hom-Poisson algebras, which is interpreted in terms of matched pairs of admissible Hom-Poisson algebras. We would like to mention that there is a theory of Hom-Poisson algebras introduced by Makhlouf and Silvestrov in \cite{MS1}  which combines a commutative Hom-associative algebra and a
Hom-Lie algebra such that a Hom-Leibniz identity is satisfied. The purpose of this paper is to develop a theory of admissible Hom-Poisson bialgebras.\\

There is a close relationship between admissible Hom-Poisson algebra and admissible Hom-pre-Poisson algebra. Namely, an admissible Hom-pre-Poisson algebra $(\mathcal{P},\prec,\succ,\alpha)$ gives rise to admissible Hom-Poisson algebra via suitable relathionship which is called the subadjacent admissible Hom-Poisson algebra denoted $\mathcal{P}^{c}$. Furthermore, define the map $L:\mathcal{P}\longrightarrow \mathfrak{gl}(\mathcal{P})$ defined by $L_{x}(y)=x\prec y, \forall x,y\in \mathcal{P}$, gives rise to a representation of the subadjacent admissible Hom-Poisson algebra on $\mathcal{P}$ with respect to $\alpha\in \mathfrak{gl}(\mathcal{P})$. 

The paper is organized as follows:
In Section ~\ref{sp} 
we deal with admissible Hom-Poisson algebras. In Section ~\ref{rep} we consider representations, matched pairs and Manin triples of admissible Hom-Poisson algebras. Section ~\ref{BI} is devoted to study the notion of admissible Hom-Poisson bialgebras with emphasis on its compatibility with Manin triples associated to nondegenrate symmetric bilinear form. In Section 5 we consider a special class of admissible Hom-Poisson bialgebra, the coboundary class. In Section ~\ref{oper} we study Hom-$\mathcal{O}$-operators of admissible Hom-Poisson bialgebras and discuss their relationship with admissible Hom-pre-Poisson algebras.

Throughout this paper, all vector spaces are finite-dimensional
over a base field $\mathbb F$.

\section{Admissible Hom-Poisson algebras}\label{sp}

In this paper we adopte the following conventions and notations
\begin{enumerate}
\item[(1)] Let $V$ be a vector space. Let $\tau:V\otimes V\to
V\otimes V$ be the flip operator defined as
    \begin{equation}
    \tau(u\otimes v)=v\otimes u, \quad \forall u,v \in V.
    \end{equation}
\item[(2)] Let $V_1,V_2$ be two vector spaces and $T:V_1\to V_2$
be a linear map. Denote the dual map by $T^*:V_2^*\to
V_1^*$, defined by
    \begin{equation}
    \langle v_1,T^*(v_2^*)\rangle =\langle T(v_1),v_2^*\rangle,\quad \forall v_1 \in V_1, v_2^* \in V_2^*.
    \end{equation}
\item[(3)] Let $V$ be a vector space and $A$ be a vector space (usually with some bilinear operations). For a linear map $\rho:A\to
{\rm End}_{\mathbb F}(V)$, define a linear map $\rho^*:A\to {\rm
End}_{\mathbb F}(V^*)$ by
    \begin{equation}
    \langle \rho^*(x)v^*,u\rangle=-\langle v^*,\rho(x)u\rangle, \quad\forall x\in A,u\in V,v^* \in V^*.
    \end{equation}
\end{enumerate}

In this paper, we recall some basics on Hom-Poisson algebras and admissible Hom-Poisson algebra introduced by Donald Yau in \cite{DY}. Recall that a Hom-Lie algebra is a triple $(A,[\cdot,\cdot],\alpha)$ where
   $A$ is a vector space, $[\cdot,\cdot]$ a bilinear map and $\alpha$ a linear map of $A$ and 
$\forall x,y \in A,$ \begin{align*}
    &[x,y]=-[y,x] \quad\text{(skew-symmetry)},\\
    &[\alpha(x),[y,z]]+[\alpha(y),[z,x]]+[\alpha(z),[x,y]]=0,\quad \text{(Jacobi identity)}.
\end{align*}
A Hom-Lie algebra $(A,[\cdot,\cdot],\alpha)$ is called multiplicative, if $\alpha([x,y]))=[\alpha(x),\alpha(y)]$, for all $x,y\in A.$ Throughout this paper all Hom-algebras is considered multiplicative.
A Hom-associative algebra  a triple
$(A,\mu,\alpha)$ consists of a vector space $A$ together with a bilinear map
$\mu: A\otimes A\longrightarrow A$ and a linear map $\alpha:A\longrightarrow A$ such that the following conditions are satisfied $\forall x,y,z\in A$
\begin{align}
&\alpha(\mu(x,y))=\mu(\alpha(x),\alpha(y))\\
&\alpha(y)\mu(\alpha(x),\mu(y,z))=\mu(\mu(y,z),\alpha(z))
\end{align}
\begin{defi} \cite{DY} The quadruple
 $(A,[\cdot,\cdot],\circ,\alpha)$ is called a Hom-Poisson algebra
if  $(A,[\cdot,\cdot],\alpha)$ is a Hom-Lie algebra,  $(A,\circ,\alpha)$ is a Hom-associative algebra and
 the following condition is satisfied
 \begin{equation}\label{lHom-poiss-jacobi}
[\alpha(x),y\circ z]=[x,y]\circ \alpha(z)+\alpha(y)\circ [x,z],\quad \forall x,y,z\in A.
\end{equation}
The operations $[\cdot,\cdot]$ and $\circ$ are called the Hom-Poisson bracket and the Hom-Poisson product respectively. When $\alpha$ is the identity map, $A$ is reduced to a classical Poisson algebra.

\end{defi}

\begin{defi}
Let $A$ be a vector space equipped with one bilinear operation
$\diamond:A\otimes A\rightarrow A$ and a morphism $\alpha:A\longrightarrow A$. We call $(A,\diamond,\alpha)$ an {\bf admissible
Hom-Poisson algebra} if the following equation holds, for all $ x,y,z\in A$,
\begin{equation}
(x\diamond y)\diamond \alpha(z)=\alpha(x)\diamond(y\diamond z)-\frac{1}{3}(-\alpha(x)\diamond (z\diamond y)+\alpha(z)\diamond( x\diamond y)+\alpha(y)\diamond (x\diamond
z)- \alpha(y)\diamond(z\diamond x)).
\label{c1}
\end{equation}
\end{defi}
\begin{rmk}
 For $\alpha=Id$ we cover the class of admissible Poisson algebra \cite{Bai3}.   An admissible Hom-Poisson algebra is called multiplicative if
$$\alpha(x\diamond y)=\alpha(x)\diamond \alpha(y)$$
it is called regular if $\alpha$ is invertible.
\end{rmk}

Here we give the Yau twist of admissible Poisson algebras
\begin{pdef}
Let $(A,\diamond)$ be an admissible Poisson algebra and let $\alpha:A\longrightarrow A$ a morphism. Then $(A,\diamond_{\alpha}:=\alpha\circ \diamond,\alpha)$ is an admissible Hom-Poisson algebra called the Yau twist of $(A,\diamond)$.    
\end{pdef}
\begin{proof}
Let $(A,\diamond)$ be an admissible Poisson algebra. Our aim is to prove that $(A,\diamond_{\alpha}:=\alpha\circ \diamond,\alpha)$ is an admissible Hom-Poisson algebra. We have
\begin{align*}
&-(x\diamond_{\alpha} y)\diamond_{\alpha} \alpha(z)+\alpha(x)\diamond_{\alpha}(y\diamond_{\alpha} z)-\frac{1}{3}(-\alpha(x)\diamond_{\alpha} (z\diamond_{\alpha} y)+\alpha(z)\diamond_{\alpha}( x\diamond_{\alpha} y)+\alpha(y)\diamond_{\alpha} (x\diamond_{\alpha}
z)\\
&- \alpha(y)\diamond_{\alpha}(z\diamond_{\alpha} x))\\
&=\alpha^{2}\Big(-(x\diamond y)\diamond z+x\diamond(y\diamond z)-\frac{1}{3}(-x\diamond (z\diamond y)+z\diamond( x\diamond y)+y\diamond (x\diamond
z)- y\diamond(z\diamond x))\Big)\\
&=0
\end{align*}
Hence the proof is acheived.
\end{proof}
\begin{defi}
An admissible Hom-Poisson algebra
$(A,\diamond,\alpha)$ is called of admissible Poisson  Hom-type if there exists an admissible Poisson algebra $(A,\diamond')$ such that $\diamond=\alpha\circ\diamond'$    \end{defi}
\begin{lem}
Any regular admissible Hom-Poisson algebra is an admissible Poisson-type. \end{lem}

\begin{pro}{\rm (\cite{DY})} \label{Yau}
If $(A,[\cdot,\cdot],\circ,\alpha)$ is a
Hom-Poisson algebra, then $(A,\diamond, \alpha)$ is an admissible Hom-Poisson algebra,
which is called the {\bf corresponding admissible Hom-Poisson
algebra}, where the multiplication $\diamond$ is defined by
\begin{equation}\label{eq:PA-sPA}
x\diamond y=x\circ y+[x,y],\;\;\forall x,y\in A.\end{equation}
Conversely, if $(A,\diamond,\alpha)$ is an admissible Hom-Poisson algebra, then
$(A,[\;,\;],\circ,\alpha)$ is a Hom-Poisson algebra, which is called the {\bf
corresponding Hom-Poisson algebra}, where the product $.$
and the bracket operation $[\;,\;]$ are respectively defined by
\begin{equation}\label{eq:sPA-PA}
 x\circ
y=\frac{1}{2}(x\diamond y+y\diamond x),\;\;\; [x,y]=\frac{1}{2}(x\diamond y-y\diamond
x),\;\;\forall x,y\in A.
\end{equation}
\end{pro}
\begin{lem}
Let $(A,\diamond,\alpha)$ be an admissible Hom-Poisson algebra, then
\begin{equation}\label{PoissAsso}
 (x\diamond y)\diamond \alpha(z)-
 \alpha(x)\diamond(y\diamond z)=\alpha(z)\diamond(y\diamond x)-
 (z\diamond y)\diamond\alpha(x).\end{equation}
\end{lem}
    

\section{Representations  and matched pairs of admissible Hom-Poisson algebras} \label{rep}
In this section, we introduce the notions of representations and
matched pairs of admissible Hom-Poisson algebras and then give some
properties.
\subsection{Representation and duale representation of Admissible Hom-Poisson algebras}
\begin{defi}\label{definition module}
Let $(A,\diamond,\alpha)$ be a multiplicative admissible Hom-Poisson algebra,  $V$ be a vector
space and $\beta\in \rm End(V)$. Let ${\frkl},\frkr:A\to {\rm End}_{\mathbb F}(V)$ be two linear maps.
The triple $(V,\frkl,\frkr,\beta)$ is called a \textup{\textbf{representation}}
of $(A,\diamond,\alpha)$ if
\begin{eqnarray}
\frkl(\alpha(x))\circ\beta=\beta\circ\frkl(x) \label{condmultipli1}\\
\frkr(\alpha(x))\circ\beta=\beta\circ\frkr(x)\label{condmultipli2}
\end{eqnarray}
{\small\begin{eqnarray}
\frkl(x\diamond y)\circ\beta=\frkl(\alpha(x))\frkl(y)-\frac{1}{3}\big(-\frkl(\alpha(x))\frkr(y)+\frkr(x\diamond y)\circ\beta+\frkl(\alpha(y))\frkl(x)-\frkl(\alpha(y))\frkr(x)\big),\label{c2} \\
\frkr(\alpha(y))\frkl(x)=\frkl(\alpha(x))\frkr(y)-\frac{1}{3}\big(-\frkl(\alpha(x))\frkl(y)+\frkl(\alpha(y))\frkl(x)+\frkr(x\diamond y)\circ\beta-\frkr(y\diamond x)\circ\beta\big),\label{c3}\\
\frkr(\alpha(y))\frkr(x)=\frkr(x\diamond y)\circ\beta-\frac{1}{3}\big(-\frkr(y\diamond x)\circ\beta+\frkl(\alpha(y))\frkr(x)+\frkl(\alpha(x))\frkr(y)-\frkl(\alpha(x))\frkl(y)\big)\label{c4}
\end{eqnarray}}
for all $x,y\in A$. Two representations $(V_1,\frkl_1,\frkr_1,\beta_{1})$ and
$(V_2,\frkl_2,\frkr_2,\beta_2)$ of an admissible Hom-Poisson algebra $A$ are called {\bf
equivalent} if there exists an isomorphism $\varphi:V_1\rightarrow
V_2$ satisfying
\begin{equation}\label{lem:rep1}
\varphi\circ \frkl_1(x) =\frkl_2(x)\circ\varphi,\quad \varphi
\frkr_1(x)=\frkr_2(x)\varphi,\quad \varphi\circ\beta_1=\beta_2\circ\varphi,\quad\;\forall x\in A.
\end{equation}
\end{defi}

\begin{lem}\label{lem:rep-property1}
  Let $(V,\frkl,\frkr,\beta)$ be a representation of an admissible Hom-Poisson algebra $(A,\diamond)$.
  Then the following equation holds:
  \begin{equation}\label{eq:rep-property1}
    \frkl(x\diamond y)\circ\beta+\frkr(\alpha(x)) \frkr(y)=\frkl(\alpha(x))\frkl(y)+\frkr(y\diamond x)\circ \beta,\;\;\forall x,y\in
    A.
  \end{equation}
\end{lem}
\begin{proof}
 By Eq. \eqref{c2}, we have
 $$\frkl(x\diamond y)\circ\beta=\frkl(\alpha(x))\frkl(y)-\frac{1}{3}\big(-\frkl(\alpha(x))\frkr(y)+\frkr(x\diamond y)\circ\beta+\frkl(\alpha(y))\frkl(x)-\frkl(\alpha(y))\frkr(x)\big) $$ and Eq. \eqref{c4} gives
$$\frkr(\alpha(x))\frkr(y)=\frkr(y\diamond x)\circ\beta-\frac{1}{3}\big(-\frkr(x\diamond y)\circ\beta+\frkl(\alpha(x))\frkl(y)+\frkl(\alpha(y))\frkr(x)-\frkl(\alpha(y))\frkl(x)\big).$$
Hence, \eqref{eq:rep-property1} holds by direct summation.
\end{proof}

\begin{pro}
Let  $(V, \rho, \mu)$ be a representation of an admissible Poisson algebra $(A,\diamond)$. Suppose that there exists $\beta:V\longrightarrow V$ such that  \ref{condmultipli1} and \ref{condmultipli2} are satisfied.\\ Then $(V,\rho_{\beta}:=\beta\circ\rho,\mu_{\beta}:=\beta\circ\mu,\beta)$ is a representation of $(A,\diamond_{\alpha},\alpha)$.   
\end{pro}
\begin{pro}
 Let $(A,\diamond,\alpha)$ be an admissible Hom-Poisson algebra with induced admissible Poisson algebra
 $(A,\diamond')$. Two representation $(V_1,\rho_{1},\mu_{1},\beta_1)$ and 
 $(V_2,\rho_{2},\mu_{2},\beta_2)$ are isomorphic if and only if the induced representations $(V_1,\rho_{1}',\mu_{1}')$ and $(V_2,\rho_{2}',\mu_{2}')$ are isomorphic by an isomorphism $f:V_1\longrightarrow V_2$.
\end{pro}
\begin{pro}
Let $(A,\diamond,\alpha)$ be an admissible Hom-Poisson algebra. Let $(V,\frkl,\frkr,\beta)$ be a representation of $A$. Define a
bilinear operation
$\diamond_{\frkl,\frkr}:(A\oplus V)\otimes (A\oplus V)\to
A\oplus V$ on $A\oplus V$  and a linear map $\alpha\oplus\beta$ as follows \begin{align*}
    &(x+u)\diamond_{\frkl,\frkr}(y+v)=x\diamond y+\frkl(x)v+\frkr(y)u,\\
    &(\alpha\oplus\beta)(x+v)=\alpha(x)+\beta(v),\quad\forall x,y\in A,u,v\in V.
    \end{align*}
Then $(V,\frkl,\frkr,\beta)$ is a representation of $(A,\diamond,\alpha)$ if and only if
$(A\oplus V,\diamond_{\frkl,\frkr},\alpha\oplus\beta)$ is an admissible Hom-Poisson algebra, which is
called the {\bf semi-direct product} of $A$ by $V$ and denoted by
$A\ltimes_{\frkl,\frkr} V$ or simply $A\ltimes V$.
\end{pro}
\begin{proof}
 Let $(V,\frkl,\frkr,\beta)$ be a representation of $(A,\diamond,\alpha)$.
Define the following operation and linear map $\alpha\oplus\beta$ on $A\oplus V$ as follows, $\forall x,y\in A,u,v\in V$\begin{align*}
    &(x+u)\diamond_{\frkl,\frkr}(y+v)=x\diamond y+\frkl(x)v+\frkr(y)u,\\
    &(\alpha\oplus\beta)(x+v)=\alpha(x)+\beta(v).
    \end{align*}
Consider $(x+u),(y+v),(z+w)\in A\oplus V$, we have
\begin{align*}
&(\alpha+\beta)(x+u)\diamond_{\frkl,\frkr}((y+v)\diamond_{\frkl,\frkr}(z+w))-\frac{1}{3}[-(\alpha+\beta)(x+u)\diamond_{\frkl,\frkr}((z+w)\diamond_{\frkl,\frkr}(y+v))\\
&+(\alpha+\beta)(z+w)\diamond_{\frkl,\frkr}((x+u)\diamond_{\frkl,\frkr}(y+v))+(\alpha+\beta)(y+v)\diamond_{\frkl,\frkr}((x+u)\diamond_{\frkl,\frkr}(z+w))\\
&-(\alpha+\beta)(y+v)\diamond_{\frkl,\frkr}((z+w)\diamond_{\frkl,\frkr}(x+u))]\\
&=\alpha(x)\diamond(y\diamond z)+\frkl(\alpha(x))\frkl(y)w+\frkl(\alpha(x))\frkr(z)v+\frkr(y\diamond z)\beta(u)-\frac{1}{3}\Big[-\alpha(x)\diamond(z\diamond y)\\
&-\frkl(\alpha(x)\frkl(z)v
-\frkl(\alpha(x)\frkr(y)w-\frkr(z\diamond y)\beta(u)+\alpha(z)\diamond(x\diamond y)+\frkl(\alpha(z)\frkl(x)v\\
&+\frkl(\alpha(z)\frkr(y)u+\frkr(x\diamond y)\beta(w)-\alpha(y)\diamond(z\diamond x)-\frkl(\alpha(y))\frkl(z)u-\frkl(\alpha(y))\frkr(x)w\\
&-\frkr(z\diamond x)\beta(v)+\alpha(y)\diamond(x\diamond z)+\frkl(\alpha(y))\frkl(x)w+\frkl(\alpha(y))\frkr(z)u+\frkr(x\diamond z)\beta(v)\Big]\\
&=\Big(\alpha(x)\diamond(y\diamond z)-\frac{1}{3}[-\alpha(x)\diamond(z\diamond y)+\alpha(z)\diamond(x\diamond y)-\alpha(y)\diamond(z\diamond x)+\alpha(y)\diamond(x\c z)]\Big)\\
&+\Big(\frkl(\alpha(x))\frkl(y)w-\frac{1}{3}[-\frkl(\alpha(x)\frkr(y)w+\frkr(x\c y)\beta(w)-\frkl(\alpha(y))\frkr(x)w+\frkl(\alpha(y))\frkl(x)w]\Big)\\
&+\Big(\frkl(\alpha(x))\frkr(z)v-\frac{1}{3}[
-\frkl(\alpha(x)\frkl(z)v+\frkl(\alpha(z))
\frkl(x)v-\frkr(z\diamond x)\beta(v)+\frkr(x\diamond z)\beta(v)]\Big)\\
&+\Big(\frkr(y\diamond z)\beta(u)-\frac{1}{3}[-\frkr(z\diamond y)\beta(u)+\frkl(\alpha(z)\frkr(y)u
-\frkl(\alpha(y))\frkl(z)u+\frkl(\alpha(y))\frkr(z)u]\Big)\\
&=(x\diamond y)\diamond \alpha(z)+\frkl(x\diamond y)\beta(w)+ \frkr(\alpha(z))\frkl(x)v+\frkr(\alpha(z))\frkr(y)u\\
&=(x\diamond y)\diamond \alpha(z)+\frkl(x\diamond y)\beta(w)+ \frkr(\alpha(z))(\frkl(x)v+\frkr(y)u)\\
&=\Big((x\diamond y)+\frkl(x)v+\frkr(y)u\Big)\diamond_{\frkl,\frkr}(\alpha(z)+\beta(w))\\
&=((x+u)\diamond_{\frkl,\frkr}(y+v))\diamond_{\frkl,\frkr}(\alpha+\beta)(z+w).
\end{align*}
Hence $(A\oplus V,\diamond_{\frkl,\frkr},\alpha\oplus\beta)$ is an admissible Hom-Poisson algebra.
\end{proof}
Let $(V,\frkl,\frkr,\beta)$ be a representation of the Hom-Poisson algebra. Always assume that $\beta$ is invertible. For all $x\in A, v\in V, \xi \in V^*$, define $\frkl^*,\frkr^*:A\longrightarrow End(V^*)$ as usaul by
\begin{align*}
&\langle\frkl^*(x)(\xi),v\rangle=-\langle \xi,\frkl(x)v\rangle.\\
&\langle\frkr^*(x)(\xi),v\rangle=-\langle \xi,\frkr(x)v\rangle.
\end{align*}
Note that $\rho^*$ is not a representation in general, one need to add strong condition, see \cite{CS} for instance, where the authors define and explain the origin of the representation $\frkl^{\c}:A\longrightarrow \mathfrak{gl}(V^*)$ where
$\frkl^{\c}(x)(\xi):=\frkl^{*}(\alpha(x))((\beta^{-2})^*(\xi)), \forall x\in A$
and how it can be obtained.

Define $\frkl^{\c},\frkr^{\c}:A\longrightarrow End(V^*)$ by
\begin{align*}
&\frkl^{\c}(x)\xi:=\frkl^{*}(\alpha(x))((\beta^{-2})^*(\xi)) ,\\
&\frkr^{\c}(x)\xi:=\frkr^{*}(\alpha(x))((\beta^{-2})^*(\xi)).
\end{align*}
\begin{pro}
Let $(A,\diamond,\alpha)$ be an admissible Hom-Poisson algebra and $(V,\frkl,\frkr,\beta)$ a representation of $A$ with $\beta$ invertible. Then, $(V^*,-\frkr^{\c},-\frkl^{\c},(\beta^{-1})^*)$ is a representation of $(A,\diamond,\alpha)$.
\end{pro}
\begin{proof}
Let $(V, \rho, \mu, \beta)$ be a representation of the admissible Hom-Poisson algebra $(A,\diamond,\alpha)$.
Our aim is to prove that $(V^*,-\frkr^{\c},-\frkl^{\c},(\beta^{-1}))$ is a representation of $(A,\diamond,\alpha)$\\
To begin with, we have $\forall x\in A,\xi\in V^*$,
\begin{align*}
&\frkr^{\c}(\alpha(x))(\beta^{-1})^{*}(\xi)=\frkr^{*}(\alpha^{2}(x))(\beta^{-3})^{*}(\xi)\\
&=\frkr^{*}(\alpha(\alpha(x)))(\beta^{-2})^{*}(\beta^{-1})^{*}(\xi)\\
&=(\beta^{-1})^{*}\frkr^{*}(\alpha(x))(\beta^{-2})^{*}(\xi)\\
&=(\beta^{-1})^{*}\frkr^{\c}(x)\xi.
\end{align*}
same as for $\frkl^{\c}$ (i.e.$\frkl^{\c}(\alpha(x))\xi:=(\beta^{-1})^{*}\frkl^{\c}(x)\xi$).\\
We have:
$\frkl(x\diamond y)\circ\beta=-\frkr(\alpha(x)) \frkr(y)+\frkl(\alpha(x))\frkl(y)+\frkr(y\diamond x)\circ \beta,\;\;\forall x,y\in
    A$\\
    Similarly we  have
$\frkr(x\diamond y)\circ \beta=\frkl(y\diamond x)\circ\beta+\frkr(\alpha(y)) \frkr(x)-\frkl(\alpha(y))\frkl(x)$.\\
So our aim becomes to prove the following identity
$$-\frkr^{\c}(x\diamond y)(\beta^{-1})^{*}=-\frkl^{\c}(y\diamond x)(\beta^{-1})^{*}-\frkr^{\c}(\alpha(x))\frkr^{\c}(y)+\frkl^{\c}(\alpha(x))\frkl^{\c}(y),$$
we have
\begin{eqnarray*}
&&\langle-\frkr^{\c}(x\diamond y)(\beta^{*})^{-1},v\rangle\\&=&\langle-\frkr^{*}(\alpha(x\diamond y)(\beta^{*})^{-3},v\rangle\\
&=&\langle (\beta^{-3})^{*},\frkr(\alpha(x\diamond y)\rangle\\
&=&\langle (\beta^{-3})^{*},\frkl(\alpha^2(y\diamond x))
+\frkr(\alpha^{2}(y))\frkr(\alpha(x))\beta^{-1}-\frkl(\alpha^{2}(y))\frkl(\alpha(x))\beta^{-1})v\rangle\\
&=&\langle (\beta^{-4})^*,\frkl(\alpha^{2}(y\diamond x))\beta+\frkr(\alpha^{3}(y))\frkr(\alpha^2(x))-\frkl(\alpha^{3}(y))\frkl(\alpha^2(x))v\rangle\\
&=&\langle (-\frkl^{*}(\alpha^{2}(y\diamond x))(\beta^{-1})^*-\frkr^{*}(\alpha^2(x))\frkr^{*}(\alpha^{3}(y))+\frkl^{*}(\alpha^2(x))\frkl^{*}(\alpha^{3}(y)))(\beta^{-4})^*,v\rangle\\
&=&\langle -\frkl^{\c}(y\diamond x)(\beta^{-1})^{*}-\frkr^{\c}(\alpha(x))\frkr^{\c}(y)+\frkl^{\c}(\alpha(x))\frkl^{\c}(y),v\rangle.
\end{eqnarray*}
Hence the proof is achieved.
\end{proof}
\begin{ex}
Let $(A,\diamond,\alpha)$ be an admissible Hom-Poisson algebra.
Let $L,R$ denote the left and right multiplication operators
$L(x)y=xy, R(x)y=yx, \forall x,y\in A$. $(A,L,R,\alpha_{A})$ is a representation of $(A,\diamond,\alpha)$ called the adjoint representation. Furthermore, $(A^*,-R^*,-L^*,\alpha_{A})$ is a representation called the coadjoint representation of $A$.
\end{ex}
\subsection{Matched pairs of admissible Hom-Poisson algebras}
\begin{defi}
 Let $(A_1,\diamond_{1},\alpha_{1})$ and $(A_2,\diamond_{2},\alpha_{2})$ be two admissible Hom-Poisson algebras and consider the maps $\frkl_{1},\frkr_{1}:A_1\longrightarrow \rm End(A_2), \quad\frkl_{2},\frkr_{2}:A_2\longrightarrow \rm End(A_1)$ \\
$(A_1,A_2,\frkl_{1},\frkr_{1},\frkl_{2},\frkr_{2})$ is called a matched pair of admissible Hom-Poisson algebras if
$(A_2,\frkl_{1},\frkr_{1},\alpha_2)$ is a representation of
$(A_1,\diamond_{1},\alpha_{1})$ and  $(A_1,\frkl_{2},\frkr_{2},\alpha_1)$ is a representation of $(A_2,\diamond_{2},\alpha_{2})$ and the following equalities hold

{\small\begin{align}
\frkr_2(\alpha_{2}(a))(x\diamond_1y)=&\frkr_2(\frkl_1(y)a)(\alpha_{1}(x))+\alpha_{1}(x)\diamond_1(\frkr_2(a)y)+\frac{1}{3}\Big(\frkr_2(\frkr_1(y)a)(\alpha_{1}(x))+\alpha_{1}(x)\diamond_1(\frkl_2(a)y) \nonumber \\
&-\frkl_2(a)(\alpha_{1}(x)\diamond_1y)-\alpha_{1}(y)\diamond_1(\frkr_2(a)x)-\frkr_2(\frkl_1(x)a)(\alpha_{1}(y))+\alpha_{1}(y)\diamond_1(\frkl_2(a)x)\nonumber \\
&+\frkr_2(\frkr_1(x)a)\alpha_{1}(y)\Big),\label{Defmatch1}\\
\frkl_2(\alpha_{2}(a))(x\diamond_1y)=&(\frkl_2(a)x)\diamond_1\alpha_{1}(y)+\frkl_2(\frkr_1(x)a)(\alpha_{1}(y))+\frac{1}{3}\Big(-\frkl_2(a)(\alpha_{1}(y)\diamond_1x)+\alpha_{1}(y)\diamond_1(\frkl_2(a)x) \nonumber \\
&+\frkr_2(\frkr_1(x)a)(\alpha_{1}(y))+\alpha_{1}(x)\diamond_1(\frkl_2(a)y)+\frkr_2(\frkr_1(y)a)(\alpha_{1}(x))-\alpha_{1}(x)\diamond_1(\frkr_2(a)y)\nonumber \\
&-\frkr_2(\frkl_1(y)a)(\alpha_{1}(x))\Big),\label{Defmatch2}\\
(\frkr_2(a)x)\diamond_1\alpha_{1}(y)=&-\frkl_2(\frkl_1(x)a)(\alpha_{1}(y))+\alpha_{1}(x)\diamond_1(\frkl_2(a)y)+\frkr_2(\frkr_1(y)a)(\alpha_{1}(x))+\frac{1}{3}\Big(\alpha_{1}(x)\diamond_1(\frkr_2(a)y)\nonumber \\
&+\frkr_2(\frkl_1(y)a)(\alpha_{1}(x))-\alpha_{1}(y)\diamond_1(\frkr_2(a)x)-\frkr_2(\frkl_1(x)a)(\alpha_{1}(y))-\frkl_2(\alpha_{2}(a))(x\diamond_1y)\nonumber \\
&+\frkl_2(\alpha_{2}(a))(y\diamond_1x)\Big),\label{Defmatch3}\\
\frkr_1(\alpha_{1}(x))(a\diamond_2b)=&\frkr_1(\frkl_2(b)x)(\alpha_{2}(a))+\alpha_{2}(a)\diamond_2(\frkr_1(x)b)+\frac{1}{3}\Big(\frkr_1(\frkr_2(b)x)(\alpha_{2}(a))+\alpha_{2}(a)\diamond_2(\frkl_1(x)b) \nonumber \\
&-\frkl_1(x)(\alpha_{2}(a)\diamond_2b)-\alpha_{2}(b)\diamond_2(\frkr_1(x)a)-\frkr_1(\frkl_2(a)x)(\alpha_{2}(b)+\alpha_{2}(b)\diamond_2(\frkl_1(x)a)\nonumber \\
&+\frkr_1(\frkr_2(a)x)(\alpha_{2}(b))\Big),\label{Defmatch4}\\
\frkl_1(\alpha_{1}(x))(a\diamond_2b)=&(\frkl_1(x)a)\diamond_2\alpha_{2}(b)+\frkl_1(\frkr_2(a)x)(\alpha_{2}(b))+\frac{1}{3}\Big(-\frkl_1(\alpha_{1}(x))(b\diamond_2a)+\alpha_{2}(b)\diamond_2(\frkl_1(x)a) \nonumber \\
&+\frkr_1(\frkr_2(a)x)(\alpha_{2}(b))+\alpha_{2}(a)\diamond_2(\frkl_1(x)b)+\frkr_1(\frkr_2(b)x)(\alpha_{2}(a))-\alpha_{2}(a)\diamond_2(\frkr_1(x)b)\nonumber \\
&-\frkr_1(\frkl_2(b)x)(\alpha_{2}(a))\Big),\label{Defmatch5}\\
(\frkr_1(x)a)\diamond_2\alpha_{2}(b)=&-\frkl_1(\frkl_2(a)x)(\alpha_{2}(b)) +\alpha_{2}(a)\diamond_2(\frkl_1(x)b)+\frkr_1(\frkr_2(b)x)(\alpha_{2}(a))+\frac{1}{3}\Big(\alpha_{2}(a)\diamond_2(\frkr_1(x)b)\nonumber \\
&+\frkr_1(\frkl_2(b)x)(\alpha_{2}(a))-\alpha_{2}(b)\diamond_2(\frkr_1(x)a)-\frkr_1(\frkl_2(a)x)(\alpha_{2}(b))-\frkl_1(\alpha_{1}(x))(a\diamond_2b)\nonumber \\
&+\frkl_1(\alpha_{1}(x))(b\diamond_2a)\Big).\label{Defmatch6}
\end{align}}
\end{defi}
By a straightforward calcultaion we get the following
\begin{pro}\label{pro:matched pair}
Let $(A_1,\diamond_1,\alpha_1)$ and $(A_2,\diamond_2,\alpha_2)$ be two admissible Hom-Poisson
algebras,$\frkl_1,\frkr_1:A_1\to {\rm End}_{\mathbb F}(A_2)$,
$\frkl_2,\frkr_2:A_2\to {\rm End}_{\mathbb F}(A_1)$ be four linear maps.
Define a bilinear operation $\diamond:(A_1\oplus A_2)\otimes
(A_1\oplus A_2)\to A_1\oplus A_2$ on $A_1\oplus A_2$ and a linear map $(\alpha_1\oplus \alpha_2)$ by
\begin{align}
&(x+a)\diamond(y+b)=x\diamond_1y+\frkr_2(b)x
+\frkl_2(a)y+\frkl_1(x)b+\frkr_1(y)a+a\diamond_2b,\\
&(\alpha_1\oplus \alpha_2)(x+a)=\alpha_1(x)+\alpha_2(a)
\end{align}
where $x,y\in A_1,a,b\in A_2.$Then
$(A_1\oplus A_2,\diamond,\alpha_1\oplus \alpha_2)$ is a matched pair of admissible Hom-Poisson algebras if and only if $(A_1,A_2,\frkl_1,\frkr_1,\frkl_{2},\frkr_{2},\alpha_1,\alpha_2)$ is a matched pair of admissible Hom-Poisson algebras. We denote it by $A_1\bowtie A_2$.
\end{pro}
Notice the semi-direct product of an admissible Hom-Poisson algebra
$(A_1,\diamond,\alpha_{1})$ by a representation $(V,\frkl,\frkr,\beta)$ is a special case of matched pair of admissible Hom-Poisson algebras
when $A_2=V$ equipped with the zero multiplication.
\section{Admissible Hom-Poisson Bialgebras}\label{BI}
In this section we consider Manin triples of admissible Hom-Poisson algebras and bialgebras. The equivalence between them is intrepreted in terms of matched pairs of admissible Hom-Poisson algebras.
\begin{pro}
  Let $(A,\diamond,\alpha_{A})$ be an admissible Hom-Poisson algebra. If there exists a nondegenerate symmetric invariant bilinear form $B$ on $A$, then the two representations $(A,L,R,\alpha_{A})$ and $(A^*,-R^*,-L^*,\alpha_{A^*})$ of the admissible Hom-Poisson $(A,\diamond,\alpha)$ are equivalent. Conversely, if the two representations  $(A,L,R,\alpha_{A})$ and $(A^*,-R^*,-L^*,\alpha_{A^*})$ of the admissible Hom-Poisson $(A,\diamond,\alpha)$ are equivalent then there exists a non degenerate invariant bilinear form $B$ on $A$.
\end{pro}
\begin{defi}
Let $(A,\diamond_{A},\alpha_{A})$ be an admissible Hom-Poisson algebra and suppose that $(A^*,\diamond_{A^*},\alpha_{A^*})$ is an admissible Hom-Poisson algebra structure on the dual space $A^*$. If there exists an admissible Hom-Poisson algebra structure on the direct sum $A\oplus A^*$ of the underlying vector spaces $A$ and $A^*$ such that  $(A,\diamond_{A},\alpha_{A})$ and  $(A^*,\diamond_{A^*},\alpha_{A^*})$ are
admissible Hom-Poisson subalgebras and the following symmetric bilinear form $B_d$ on $A\oplus A^*$ by
\begin{equation}\label{bilinearform}
B_d(x+a^*,y+b^*)=\langle x,b^*\rangle +\langle y,a^*\rangle, \forall x,y\in A, a^*,b^*\in A^*
\end{equation}
 is invariant, then, $(A\oplus A^*,A,A^*)$ is called a standard Manin triple of admissible Hom-Poisson algebras assocaited to $B_d$ defined by \eqref{bilinearform}.
\end{defi}
\begin{pro}
 Let $(A,\diamond_{A},\alpha_{A})$ be an admissible Hom-Poisson algebra.  Suppose that $(A^*,\diamond_{A^*},\alpha_{A}^*)$ is an admissible Hom-Poisson algebra such that the Hom-admissible structure $"\diamond_{A^*}"$ on the dual space is given by a linear map $\Delta^*:A^*\otimes A^*\longrightarrow A^*$. Then $(A\oplus A^*,A,A^*,\alpha_{A}\oplus\alpha^*_{A^*})$ is a standard Manin triple of admissible Hom-Poisson algebras associated to $B_d$ if and only if  $(A,A^*,-R_A^{\star},-L_A^{\star},-R_{A^*}^{\star},-L_{A^*}^{\star},\alpha,\alpha^*)$ is a matched pair of admissible Hom-Poisson algebras.
\end{pro}

\begin{thm}
Let $(A,\diamond,\alpha)$ be an admissible Hom-Poisson algebra and suppose that there exists an admissible Hom-Poisson algebra structure $\c_{A^*}$ on the dual space $A^*$ given by the linear map $\Delta^*:A^*\otimes A^*\longrightarrow A^*$.Then $(A,A^*,-R_A^*,-L_A^*,-R_{A^*}^*,-L_{A^*},\alpha_{A},\alpha_{A^*})$ is a matched pair of admissible Hom-Poisson algebras if and only if $\Delta$ satisfies the following equations:
{\small\begin{align}
&\Delta(x\diamond_A y)-
 (R_A(\alpha^{-1}(y))\otimes {\rm
 \alpha})\Delta(x)-({\rm \alpha}\otimes L_A(\alpha^{-1}(x)))\Delta(y)\nonumber\\
    =&\frac{1}{3}\big((L_A(\alpha^{-1}(y))\otimes {\rm \alpha})\Delta(x)-({\rm \alpha}\otimes L_A(\alpha^{-1}(y)))\Delta(x)+(L_A(\alpha^{-1}(x))\otimes {\rm \alpha})\Delta(y)-(R_A(\alpha^{-1}(x))\otimes {\rm \alpha})\Delta(y)\nonumber\\
    &+\tau(-\Delta(x\diamond_A y)+(L_A(\alpha^{-1}(x))\otimes {\rm \alpha})\Delta(y)+(L_A(\alpha^{-1}(y))\otimes {\rm \alpha})\Delta(x))\big),\label{Defbi1}\\
    &\Delta(x\diamond_A y)-(R_A(\alpha^{-1}(x))\otimes {\rm \alpha})\Delta(y)-({\rm \alpha}\otimes L_A(\alpha^{-1}(y)))\Delta(x)\nonumber \\
    =&\frac{1}{3}\big(\tau((L_A(\alpha^{-1}(x))\otimes {\rm \alpha})\Delta(y)+(L_A(\alpha^{-1}(y))\otimes {\rm \alpha})\Delta(x)-(R_A(\alpha^{-1}(y))\otimes {\rm \alpha})\Delta(x)-({\rm \alpha}\otimes L_A(\alpha^{-1}(x)))\Delta(y))\nonumber \\
    &-\Delta(y\diamond_A x)+(L_A(\alpha^{-1}(x))\otimes {\rm \alpha})\Delta(y)+(L_A(\alpha^{-1}(y))\otimes {\rm \alpha})\Delta(x)\big) ,\label{Defbi2}\\
    &(({\rm \alpha}\otimes R_A(y))\Delta(x)-(L_A(\alpha^{-1}(y))\otimes {\rm \alpha})\Delta(x)+\tau\left(({\rm \alpha}\otimes R_A(\alpha^{-1}(x)))\Delta(y)-(L_A(\alpha^{-1}(x))\otimes {\rm \alpha})\Delta(y)\right)\nonumber \\
    =&\frac{1}{3}\big((R_A(\alpha^{-1}(y))\otimes {\rm \alpha})\Delta(x)-({\rm \alpha}\otimes L_A(\alpha^{-1}(y))\Delta(x)+({\rm \alpha}\otimes L_A(\alpha^{-1}(x))\Delta(y)-(R_A(\alpha^{-1}(x))\otimes {\rm \alpha})\Delta(y)\nonumber \\
    &+\tau(\Delta(y\diamond_A x)-\Delta(x\diamond_A y))\big)\label{Defbi3}
\end{align}}
for all $x,y\in A$.
\end{thm}
\begin{proof}
Let $\frkl_1=-R^{\star}_{A},\frkr_1=-L^{\star}_{A}, \frkl_2=-R^{\star}_{A^*}, \frkr_2=-L^{\star}_{A^*}, \alpha_1=\alpha_{A}=\alpha, \alpha_2=\alpha_{A^*}=(\alpha^{-1})^*$, we have
 \begin{eqnarray*}
{\rm Eq.}\ (\ref{Defmatch1})\Longleftrightarrow{\rm Eq.}\ (\ref{Defbi1}),\quad{\rm Eq.}\ (\ref{Defmatch2})\Longleftrightarrow{\rm Eq.}\ (\ref{Defbi2}), \quad{\rm Eq.}\ (\ref{Defmatch3})\Longleftrightarrow{\rm Eq.}\
(\ref{Defbi3}).
\end{eqnarray*}
In fact, we have for the first equivalence
${\rm Eq.}\ (\ref{Defmatch1})\Longleftrightarrow{\rm Eq.} (\ref{Defbi1})$
we have
\begin{align*}
&-L_{A^{*}}^{\star}((
\alpha^{-1})^*(a))(x\diamond_{A}y)=-L_{A^{*}}^{\star}(-R^{*}_A(y)a)\alpha(x)+\alpha(x)\diamond_{A}(-L^{\star}_{A^{*}}(a)y)+\frac{1}{3}\Big(-L^{\star}_{A^{*}}(-L^{\star}_{A}(y)a)\alpha(x)\\
&+\alpha(x)\diamond_{A}(-R^{\star}_{A^{*}}(a)y)
+R^{\star}_{A^{*}}(a)(\alpha(x)\diamond_{A}y)-\alpha(y)\diamond(-L^{\star}_{A^{*}}(a)x)+L^{\star}_{A^{*}}(-R^{\star}_{A}(x)a)\alpha(y)\\
&+\alpha(y)\diamond_{A}(-R^{\star}_{A^{*}}(a)x)-L^{\star}_{A^{*}}(-L^{\star}_{A}(x)a)\alpha(y) \Big)  \end{align*}
Recall from \cite{BY} that
\begin{eqnarray*}
L_{A^*}^{\star}(x)a&=&L_{A^*}^{*}((\alpha^{-1})^*(x))\alpha^{2}(a)\\
L_{A}^{\star}(x)a&=&L_{A}^{*}(\alpha(x))(\alpha^{-2})^*(a)
\end{eqnarray*}
Then we have
\item \begin{eqnarray*}
\langle -L^{\star}_{A^*}(\alpha_{A^*}(a))(x\diamond_{A}y),b \rangle &=&\langle -L^{\star}_{A^*}((\alpha^{-1})^*(a))x\diamond_{A}y,b \rangle\\ &=&\langle-L^{*}_{A^*}((\alpha^{-2})^*(a))\alpha^{2}( x\diamond_{A}y),b \rangle\\
&=&\langle\alpha^{2}( x\diamond_{A}y), L_{A^*}((\alpha^{-2})^*(a))b\rangle\\
&=&\langle \alpha^{2}( x\diamond_{A}y),\Delta^{*}((\alpha^{-2})^*(a)\otimes b)\rangle\\
&=&\langle \Delta(x\diamond_{A}y),a\otimes (\alpha^{2})^*(b)\rangle,
\end{eqnarray*}
 \begin{eqnarray*}
\langle -L^{\star}_{A^*}(-R^{\star}_{A}(y)a)\alpha(x),b \rangle&=& \langle L^{*}_{A^*}((\alpha^{-1})^*(R^{*}_{A}\alpha(y))(\alpha^{-2})^*(a))\alpha^{3}(x),b\rangle \\
&=&\langle L^{*}_{A^*}((R^{*}_{A}(\alpha^{2}(y))(\alpha^{-3})^*(a))\alpha^{3}(x),b\rangle\\&=-&\langle \alpha^{3}(x),L_{A^*}((R^{*}_{A}(\alpha^{2}(y))(\alpha^{-3})^*(a))b\rangle
\\&=&-\langle \alpha^{3}(x),\Delta^*((R^{*}_{A}(\alpha^{2}(y))(\alpha^{-3})^*(a))\otimes b)\rangle\\&=&-\langle x,\Delta^*((R^{*}_{A}( \alpha^{-1} (y)) (a)\otimes (\alpha^{3})^*(b))\rangle\\&=&\langle \Delta(x),(R^{*}_{A}( \alpha^{-1} (y)) (a)\otimes (\alpha^{3})^*(b)\rangle
\\
&=&\langle ( R_A(\alpha^{-1}(y))\otimes \alpha)\Delta(x) ,a\otimes (\alpha^{2})^*(b)\rangle,
\end{eqnarray*}
and
\begin{eqnarray*}
\langle\alpha(x)\diamond_A(-L^{\star}_{A^*}(a))y,b \rangle &=&\langle L_A(\alpha(x))(-L^{\star}_{A^*}(a)y),b\rangle\\
&=&\langle L_{A^*}^{\star}(a)y,L^*_{A}(\alpha(x))b \rangle\\
&=&\langle L^{*}_{A^*}((\alpha^{-1})^*(a))\alpha^{2}(y),L_{A}^*(\alpha(x))b\rangle\\
&=&-\langle \alpha^{2}(y),L_{A^*}((\alpha^{-1})^*(a))(L_{A}^*(\alpha(x))b)\rangle\\
&=&-\langle y,\Delta^{*}((\alpha^2)^*(\alpha^{-1})^*(a)\otimes
(\alpha^{2})^*(L^*_{A}(\alpha(x))b))\rangle \\
&=&-\langle y,\Delta^*(\alpha^*\otimes L_{A}^*((\alpha^{-1}(x))(a\otimes (\alpha^2)^*(b)))) \rangle\\
&=&\langle(\alpha\otimes L_{A}(\alpha^{-1}(x)))\Delta(y),a\otimes (\alpha^2)^*(b) \rangle .
\end{eqnarray*}

\end{proof}
\begin{lem}\label{lem:cosp}
Let $(A,\diamond,\alpha)$ be an admissible Hom-Poisson algebra and $\Delta: A  \rightarrow A\otimes A$
be a linear map. Then the dual map $\Delta^* : A^*\otimes A^*
\rightarrow A^*$ defines an admissbible Hom-Poisson  algebra structure on
$A^*$ if and only if $\Delta$ satisfies
\begin{eqnarray}
&&(\alpha\otimes\Delta)\Delta(x)-(\Delta\otimes\alpha)\Delta(x)+\frac{1}{3}\big(({\rm id}\otimes \tau)(\alpha\otimes \Delta)\Delta(x)-(\tau\otimes{\rm id})(\alpha\otimes\Delta)\Delta(x)\nonumber\\
      &&-({\rm id}\otimes\tau)(\tau\otimes{\rm id})(\alpha\otimes\Delta)\Delta(x)+(\tau\otimes{\rm id})({\rm id}\otimes\tau)(\alpha\otimes\Delta)\Delta(x)\big)=0,\;\;\forall x\in A.\label{eq:coalgebra}
\end{eqnarray}

\end{lem}

\begin{proof}Let $x\in A,$ we have
\begin{align*}
\langle(\alpha\otimes\Delta)\Delta(x),a\otimes b\otimes c\rangle&=\langle\Delta(x),\alpha(a)\otimes (b\diamond c)\rangle\\
&=\langle x,\alpha(a)\diamond (b\diamond c)\rangle,
\end{align*}
\begin{align*}
\langle (\Delta\otimes\alpha)\Delta(x),a\otimes b\otimes c \rangle
&=\langle \Delta(x),(a\diamond c)\otimes \alpha(c) \rangle \\
&= \langle x,(a\diamond b)\diamond\alpha(c)  \rangle ,   
\end{align*}
\begin{align*}
\langle (id\otimes \tau)(\alpha\otimes \Delta)\Delta(x),a\otimes b\otimes c \rangle &=\langle(\alpha\otimes \Delta)\Delta(x),a\otimes c\otimes b \rangle \\&= \langle \Delta(x),\alpha(a)\otimes(c\diamond b) \rangle\\&=\langle x,\alpha(a)\diamond(c\diamond b) \rangle ,    
\end{align*}
\begin{align*}
\langle(\tau\otimes{\rm id})({\rm id}\otimes\tau)(\alpha\otimes\Delta)\Delta(x)\big),a\otimes b\otimes c\rangle&=\langle({\rm id}\otimes\tau)(\alpha\otimes\Delta)\Delta(x)\big),b\otimes a\otimes c\rangle\\&=\langle(\alpha\otimes\Delta)\Delta(x)\big),b\otimes c\otimes a\rangle\\
&=\langle x,\alpha(b)\diamond (c\diamond a)\rangle.
\end{align*}
As a result, 
\begin{align*}\Big\langle 
&\Big((\alpha\otimes\Delta)\Delta(x)-(\Delta\otimes\alpha)\Delta(x)+\frac{1}{3}\big(({\rm id}\otimes \tau)(\alpha\otimes \Delta)\Delta(x)-(\tau\otimes{\rm id})(\alpha\otimes\Delta)\Delta(x)\nonumber\\
&-({\rm id}\otimes\tau)(\tau\otimes{\rm id})(\alpha\otimes\Delta)\Delta(x)+(\tau\otimes{\rm id})({\rm id}\otimes\tau)(\alpha\otimes\Delta)\Delta(x)\Big)\Big),a\otimes b\otimes c
\Big\rangle\\
&=\Big\langle x,\alpha(a)\diamond(b\diamond c)-(a\diamond b)\diamond \alpha(c)+\frac{1}{3}\Big (\alpha(a)\diamond(c\diamond b)-\alpha(b)\diamond(a\diamond c)-\alpha(b)\diamond(c\diamond a)\\
&+\alpha(c)\diamond(a\diamond c)\Big) \Big\rangle=0\ \ (\text{By } \eqref{c1}).
\end{align*}
\end{proof}
\begin{defi}
 An admissible Hom-Poisson algebra
 $(A,\diamond,\alpha)$ equipped with
 a linear map $\Delta:A\longrightarrow A\otimes A$ satisfying \ref{eq:coalgebra}
 is called an \textbf{admissible Hom-Poisson coalgebra}.
\end{defi}
\begin{defi}
An admissible Hom-Poisson bialgebra structure on a Hom-Poisson algebra
is a linear map $\Delta:A\longrightarrow A\otimes A$
satisfying Equations: \eqref{Defbi1}, \eqref{Defbi2}, \eqref{Defbi3} and  \eqref{eq:coalgebra}
\end{defi}
We get a characterization of admissible Hom-Poisson bialgebra by matched pairs in the following theorem.
\begin{thm}
 Let $(A,\diamond,\alpha)$ be an admissible Hom-Poisson algebra. Let $\Delta:A\longrightarrow A\otimes A$ be a linear map such that $(A,\Delta,\alpha)$ is an admissible Hom-Poisson coalgebra. Then the following conditions are equivalent:
 \begin{enumerate}
\item $(A,\diamond,\Delta,\alpha)$ is an admissible Hom-Poisson bialgebra.
\item $(A,A^*,-R_A^*,-L_A^*,-R_{A^*}^*,-L_{A^*},\alpha_{A},\alpha_{A^*})$ is a matched pair of admissible Hom-Poisson algebras.
\item $(A\oplus A^*,A,A^*)$ is standard Manin triple.
    \end{enumerate}
\end{thm}
Note that there is a one-to-one correspondance between admissible Hom-Poisson bialgebra structures and Hom-Poisson bialgebra structures.
\begin{pro}
Let $(A,\diamond,\Gamma,\alpha)$ be an admissible Hom-Poisson bialgebra. Then
$(A,[\cdot,\cdot],.,
\delta,\Delta,\alpha)$
is a Hom-Poisson bialgebra where
\begin{align*}
&\delta=\frac{1}{2}(\Gamma-\tau\Gamma);\quad \Delta=\frac{1}{2}(\Gamma+\tau\Gamma).
\end{align*}
Conversely, let $(A,[\cdot,\cdot],.,\delta,\Delta,\alpha)$ be a Hom-Poisson bialgebra. Define a linear map $$\Gamma=\delta+\Delta.$$
Then, $(A,\diamond,\Gamma,\alpha)$ is an admissible Hom-Poisson bialgebra.
\end{pro}
\section{Coboundary admissible Hom-Poisson algebra}
In this section we consider a special class of admissible Hom-Poisson bialgebra that is, an admissible Hom-Poisson bialgebra$(A,\diamond,\Delta,\alpha)$.
 with $\Delta$ in the form 
 \begin{equation}\label{eq:cobo}
    \Delta(x)=({\rm \alpha}\otimes L_A(\alpha^{-2}(x))-R_A(\alpha^{-2}(x))\otimes {\rm \alpha})r,\;\; \forall x\in A.
    \end{equation}
 \begin{equation}\label{cond:r}
 (\alpha\otimes\alpha)r=r
 \end{equation}
\begin{pro}\label{pro:co1}
Let $(P,\diamond,\alpha)$ be an admissible Hom-Poisson algebra and $r\in P\otimes P$. Let
$\Delta:P\rightarrow P\otimes P$ be a linear map defined by
Eq.~(\ref{eq:cobo}). Then for all $x,y\in P$, we have
\begin{enumerate}
\item $\Delta$ satisfies Eq.~(\ref{Defbi1}) if and only if
\begin{eqnarray}
&&({\rm \alpha}\otimes L(\alpha^{-1}(x)))(L(\alpha^{-2}(y))\otimes {\rm \alpha}-{\rm \alpha}\otimes
R(\alpha^{-2}(y)))(r+\tau(r))\nonumber\\ \nonumber
  &&+({\rm \alpha}\otimes L(\alpha^{-1}(y)))(L(\alpha^{-2}(x))\otimes {\rm
  \alpha}-{\rm \alpha}\otimes R(\alpha^{-2}(x))(r+\tau(r))\\
  &&-(L(\alpha^{-2}(x\diamond y))\otimes {\rm \alpha}-{\rm \alpha}\otimes R(\alpha^{-2}(x\diamond
  y)))(r+\tau(r))=0;\label{eq:eqv1}
\end{eqnarray}
\item $\Delta$ satisfies Eq.~(\ref{Defbi2}) if and only if
\begin{eqnarray}
&&({\rm \alpha}\otimes L(\alpha^{-1}(x)))(L(\alpha^{-2}(y))\otimes {\rm \alpha}-{\rm \alpha}\otimes
R(\alpha^{-2}(y)))(r+\tau(r))\nonumber\\\nonumber
  &&+({\rm \alpha}\otimes L(\alpha^{-1}(y)))(L(\alpha^{-2}(x))\otimes {\rm
  id}-{\rm id}\otimes R(x))(r+\tau(r))\\\nonumber
  &&-(L(\alpha^{-1}(x))\otimes {\rm \alpha})(L(\alpha^{-2}(y))\otimes {\rm \alpha}-{\rm \alpha}\otimes
  R(\alpha^{-2}(y)))(r+\tau(r))\\
  &&-({\rm \alpha}\otimes R(\alpha^{-1}(y)))(L(\alpha^{-2}(x))\otimes {\rm \alpha}-{\rm \alpha}\otimes
  R(\alpha^{-2}(x)))(r+\tau(r))=0;\label{eq:eqv2}
\end{eqnarray}
\item $\Delta$ satisfies Eq.~(\ref{Defbi3}) if and only if
\begin{eqnarray}
&&(R(\alpha^{-2}(x))\otimes {\rm id}-{\rm id}\otimes L(\alpha^{-2}(x)))(L(\alpha^{-2}(y))\otimes {\rm
id}-{\rm id}\otimes R(\alpha^{-2}(y)))(r+\tau(r))\nonumber\\
  &&+\frac{1}{3}\big((L(x\diamond y)-L(y\diamond x))\otimes {\rm id}-{\rm id}\otimes(R(x\diamond y)-R(y\diamond
  x))\big)(r+\tau(r))=0.\label{eq:eqv3}
\end{eqnarray}
\end{enumerate}
\end{pro}
\begin{proof}
For (a) we substitute 
\eqref{eq:cobo} in \eqref{Defbi1}
we get 
{\small\begin{align*}
0=&-(\alpha\otimes \alpha^{-2}L(x\diamond y)\alpha^2)r+(\alpha^{-2}R(x\diamond y)\alpha^2\otimes \alpha)r+(R(\alpha^{-1}(y)\alpha\otimes L(\alpha^{-1}(x)\alpha)r\\&-(\alpha^{-2}R(\alpha(y))R(x)\alpha^{2}\otimes \alpha^{2})r+(\alpha^{2}\otimes \alpha^{-2}L(\alpha(x))L(y)\alpha^2)r-(R(\alpha^{-1}(y))\alpha\otimes L(\alpha^{-1}(x))\alpha)r\\
&+\frac{1}{3}\Big[(L(\alpha^{-1}(y))\alpha\otimes L(\alpha^{-1}(x))\alpha)r-(\alpha^{-2}L(\alpha(y))R(x)\alpha^{2}\otimes \alpha^{2})r-((\alpha^{2}\otimes \alpha^{-2}L(\alpha(y))L(x)\alpha^{2})r\\
&+(R(\alpha^{-1}(x))\alpha\otimes L(\alpha^{-1}(y))\alpha)r+(L(\alpha^{-1}(x))\alpha\otimes L(\alpha^{-1}(y))\alpha)r
-(\alpha^{-2}L(\alpha(x))R(y)\alpha^2\otimes \alpha^2)r\\
&-(R(\alpha^{-1}(x))\alpha\otimes L(\alpha^{-1}(y))\alpha)r+(\alpha^{-2}(R(\alpha(x))R(y)\alpha^{2}\otimes\alpha^{2})r+(L(\alpha^{-1}(y))\alpha\otimes L(\alpha^{-1}(x))\alpha)\tau(r)\\
&-(\alpha^{-2}L(\alpha(x))R(y)\alpha^{2}\otimes \alpha^{2})\tau(r)+(L(\alpha^{-1}(x))\alpha\otimes L(\alpha^{-1}(y))\alpha)\tau(r)
-(\alpha^{-2}L(\alpha(y))R(x)\alpha^{2}\otimes \alpha^{2})\tau(r)\\
&-( \alpha^{-2}L(x\diamond y)\alpha^2\otimes \alpha)\tau(r)+(\alpha\otimes\alpha^{-2}R(x\diamond y)\alpha^2)\tau(r)
\Big] .\end{align*}}
Using the condition \eqref{cond:r}, we get
{\small\begin{align*}
&
(\alpha^2\otimes \alpha^{-2}L(x\diamond y)\circ \alpha\alpha^2)r+(\alpha^{-2}R(x\diamond y)\circ \alpha\alpha^2\otimes \alpha)r+(R(\alpha^{-1}(y)\alpha\otimes L(\alpha^{-1}(x)\alpha)r\\&-(\alpha^{-2}R(\alpha(y))R(x)\alpha^{2}\otimes \alpha^{2})r+(\alpha^{2}\otimes \alpha^{-2}L(\alpha(x))L(y)\alpha^2)r-(R(\alpha^{-1}(y))\alpha\otimes L(\alpha^{-1}(x))\alpha)r\\
&+\frac{1}{3}\Big[(L(\alpha^{-1}(y))\alpha\otimes L(\alpha^{-1}(x))\alpha)r-(\alpha^{-2}L(\alpha(y))R(x)\alpha^{2}\otimes \alpha^{2})r-((\alpha^{2}\otimes \alpha^{-2}L(\alpha(y))L(x)\alpha^{2})r\\
&+(R(\alpha^{-1}(x))\alpha\otimes L(\alpha^{-1}(y))\alpha)r+(L(\alpha^{-1}(x))\alpha\otimes L(\alpha^{-1}(y))\alpha)r
-(\alpha^{-2}L(\alpha(x))R(y)\alpha^2\otimes \alpha^2)r\\
&-(R(\alpha^{-1}(x))\alpha\otimes L(\alpha^{-1}(y))\alpha)r+(\alpha^{-2}(R(\alpha(x))R(y)\alpha^{2}\otimes\alpha^{2})r+(L(\alpha^{-1}(y))\alpha\otimes L(\alpha^{-1}(x))\alpha)\tau(r)\\
&-(\alpha^{-2}L(\alpha(x))R(y)\alpha^{2}\otimes \alpha^{2})\tau(r)+(L(\alpha^{-1}(x))\alpha\otimes L(\alpha^{-1}(y))\alpha)\tau(r)
-(\alpha^{-2}L(\alpha(y))R(x)\alpha^{2}\otimes \alpha^{2})\tau(r)\\
&-( \alpha^{-2}L(x\diamond y)\circ \alpha\alpha^2\otimes \alpha^2)\tau(r)+(\alpha^2\otimes\alpha^{-2}R(x\diamond y)\circ \alpha\alpha^2)\tau(r)
\Big].
\end{align*}}
So, we get
\begin{align*}
A_1=&\alpha^{-2}\Big[R(y\diamond x)\alpha-R(\alpha(y))R(x)+\frac{1}{3}(R(y\diamond y)\alpha-L(\alpha(y))R(x)-L(\alpha(x))R(y)\\
&+L(\alpha(x))L(y))\Big]\alpha^{2}\otimes\alpha^{2}=0    
\end{align*}
\begin{align*}
A_2=&\alpha^{2}\otimes\alpha^{-2}\Big[
 -L(x\diamond y)\alpha+L(\alpha(x))L(y)+\frac{1}{3}L(\alpha(x))R(y)-R(x\diamond y)\alpha-L(\alpha(y))L(x)\\
 &-L(\alpha(y))R(x)\Big]\alpha^{2}=0  
\end{align*}
\begin{align*}
A_3=&\Big[(L(\alpha^{-1}(y))\alpha\otimes L(\alpha^{-1}(x))\alpha)- (\alpha^{-2}L(\alpha(x))R(y)\alpha^{2}\otimes \alpha^{2})\Big](r+\tau(r))\\
+& \Big[(L(\alpha^{-1}(x))\alpha\otimes L(\alpha^{-1}(y))\alpha)- (\alpha^{-2}L(\alpha(y))R(x)\alpha^{2}\otimes \alpha^{2})\Big](r+\tau(r))\\
&-\Big[(\alpha^{-2}L(x\diamond y)\alpha\alpha^{2}\otimes\alpha^2)\Big]\tau(r)+\Big[(\alpha^{2}\otimes\alpha^{-2}R(x\diamond y)\alpha\alpha^{2}\Big]r\\
&+\Big[(\alpha^{2}\otimes\alpha^{-2}R(x\diamond y)\alpha\alpha^{2}\Big]\tau(r),
\end{align*}
\begin{align*}
A_4=&(\alpha^{-2}\Big[R(y
\diamond x)\alpha+L(\alpha(x))L(y)
-R(\alpha(x))R(y)
 \Big]\alpha^{2} \otimes\alpha^2)(r)\\
 &=(\alpha^{-2}L(x\diamond y)\alpha \alpha^2\otimes \alpha^2 ) (r).
\end{align*}
So we get finally by adding $A_4$ to $A_3$
\begin{eqnarray*}
&&({\rm \alpha}\otimes L(\alpha^{-1}(x)))(L(\alpha^{-2}(y))\otimes {\rm \alpha}-{\rm \alpha}\otimes
R(\alpha^{-2}(y)))(r+\tau(r))\nonumber\\ \nonumber
  &&+({\rm \alpha}\otimes L(\alpha^{-1}(y)))(L(\alpha^{-2}(x))\otimes {\rm
  \alpha}-{\rm \alpha}\otimes R(\alpha^{-2}(x))(r+\tau(r))\\
  &&-(L(\alpha^{-2}(x\diamond y))\otimes\alpha-\alpha\otimes R(\alpha^{-2}(x\diamond y))(r+\tau(r)=0.
\end{eqnarray*}
For (b) and (c), it is similar to (a).
\end{proof}

\section{$\mathcal O$-operators of admissible Hom-Poisson algebra admissible Hom-pre-Poisson algebras}\label{oper}
In this section, we introduce the notion of $\mathcal O$-operators of admissible Hom-Poisson algebras.On one hand Hom-$\mathcal O$-operator on admissible Hom-Poisson algebras gives rise to a Hom-pre-Poisson and on the other hand a Hom-pre-Poisson gives a Hom-$\mathcal O$-operator on the subadjacent admissible Hom-Poisson algebra. 

\begin{defi}
An admissible Hom-pre-Poisson algebra is a quadruple $(\mathcal{P},\succ,\prec,\alpha)$ such that $\mathcal{P}$ is a vector space,
$\succ,\prec:\mathcal{P}\otimes \mathcal{P}\to \mathcal{P}$ are two bilinear operations and $\alpha:\mathcal{P}\longrightarrow \mathcal{P}$ a linear map
satisfying the following conditions:
\begin{eqnarray}
A(x,y,z):&=&-(x\succ y)\succ \alpha(z)-(x\prec y)\succ \alpha(z)+\alpha(x)\succ(y\succ z)\nonumber\\
&&+\frac{1}{3}\big(\alpha(x)\succ(z\prec y)
      -\alpha(z)\prec(x\succ y)-\alpha(z)\prec(x\prec y)\nonumber\\
      &&-\alpha(y)\succ(x\succ z)+\alpha(y)\succ(z\prec x)\big)=0,\label{eq:presp1}\\
B(x,y,z):&=&-\alpha(x)\succ(z\prec y)+(x\succ z)\prec \alpha(y)+\frac{1}{3}\big(-\alpha(x)\succ(y\succ z)\nonumber\\
&&+\alpha(y)\succ(x\succ z)
      +\alpha(z)\prec(x\prec y)+\alpha(z)\prec(x\succ y)\nonumber\\
      &&-\alpha(z)\prec(y\succ x)-\alpha(z)\prec(y\prec x)\big)=0,\label{eq:presp2}\\
C(x,y,z):&=&-\alpha(z)\prec(x\succ y)-\alpha(z)\prec(x\prec y)+(z\prec x)\prec \alpha(y)\nonumber\\
     &&+\frac{1}{3}\big(-\alpha(z)\prec(y\succ x)-\alpha(z)\prec(y\prec x)\nonumber\\
     &&+\alpha(y)\succ(z\prec x)+\alpha(x)\succ(z\prec y)-\alpha(x)\succ(y\succ z)\big)=0\label{eq:presp3}
\end{eqnarray}
for all $x,y,z\in \mathcal{P}$.
\begin{rmk}An admissible Hom-pre-Poisson algebra
is called regular if $\alpha$ is invertible. An admissible pre-Poisson algebra is an admissible Hom-pre-Poisson algebra when $\alpha=id.$ \end{rmk}
\end{defi} 

Let $(\mathcal{P},\prec,\succ,\alpha)$
and $(\mathcal{P}',\prec',\succ',\alpha')$ be admissible Hom-pre-Poisson algebras. A morphism  $f:\mathcal{P} \longrightarrow\mathcal{P}'$, such that $\forall x,y\in \mathcal{P}$
\begin{align}
& f(x\prec y)=f(x)\prec'f(y),\\
&f(x\succ y)=f(x)\succ'f(y),\\
&f\circ \alpha =\alpha'\circ f.
\end{align}

In the following, we establish a relationship between admissible pre-Poisson algebra and admissible Hom-Poisson algebra using the Yau-twist (see \cite{DY} for more details).
\begin{pro}
Let $(\mathcal{P},\prec,\succ)$ be an admissible pre-Poisson algebra and let $\alpha:\mathcal{P}\longrightarrow \mathcal{P}$ be a linear map. Then $(\mathcal{P},\prec_{\alpha}=\alpha\circ \prec,\succ_{\alpha}=\alpha\circ\succ,\alpha)$ is an admissible Hom-pre-Poisson algebra called the Yau twist of $\mathcal{P}$. Moreover, suppose $(\mathcal{P}',\prec',\succ')$ is another
admissible pre-Poisson algebra and $\alpha':\mathcal{P}'\longrightarrow \mathcal{P}'$ is an admissible pre-Poisson algebra morphism that satisfies $f\circ \alpha=\alpha'\circ f,$ then $f:(\mathcal{P},\prec_{\alpha},\succ_{\alpha},\alpha)\longrightarrow (\mathcal{P}',\prec'_{\alpha'},\succ'_{\alpha'},\alpha')$ is a morphism of admissible Hom-pre-Poisson algebra
\end{pro}
\begin{proof}
 Since  $(\mathcal{P},\prec,\succ)$ is an admissible pre-Poisson algebra and $\alpha$ be a morphism, then we have
 \begin{eqnarray*}
A(x,y,z):&=&-(x\succ_{\alpha} y)\succ_{\alpha} \alpha(z)-(x\prec_{\alpha} y)\succ_{\alpha} \alpha(z)+\alpha(x)\succ_{\alpha}(y\succ_{\alpha} z)\nonumber\\
&&+\frac{1}{3}\big(\alpha(x)\succ_{\alpha}(z\prec_{\alpha} y)
      -\alpha(z)\prec_{\alpha}(x\succ_{\alpha} y)-\alpha(z)\prec_{\alpha}(x\prec_{\alpha} y)\nonumber\\
      &&-\alpha(y)\succ_{\alpha}(x\succ_{\alpha} z)+\alpha(y)\succ_{\alpha}(z\prec_{\alpha} x)\big)\\
      &&=-\alpha^{2}\Big(-(x\succ y)\succ z-(x\prec y)\succ z+x\succ(y\succ z)+\frac{1}{3}\big(x\succ(z\prec y)\nonumber\\
      &&-z\prec(x\succ y)-z\prec(x\prec y)-y\succ(x\succ z)+y\succ(z\prec x)\big)\Big)=0.
 \end{eqnarray*}
 Same work we get 
 $$B(x,y,z)=0=C(x,y,z).$$
 Therefore, $(\mathcal{P},\prec_{\alpha},\succ_{\alpha},\alpha)$ is an admissible Hom-pre-Poisson algebra.
\end{proof}
\begin{cor}
 Let $(\mathcal{P},\prec,\succ,\alpha)$
 be a regular Hom-pre-Poisson algebra.  Then $(\mathcal{P},\prec_{\alpha^{-1}}:=\alpha^{-1}\circ \prec,
 \succ_{\alpha^{-1}}:=\alpha^{-1}\circ \succ)$ is an admissible pre-Poisson algebra.
\end{cor}

There is a close relationship between admissible Hom-Poisson algebra and admissible Hom-pre-Poisson algebra. Namely, an admissible Hom-pre-Poisson algebra $(\mathcal{P},\prec,\succ,\alpha)$ gives rise to admissible Hom-Poisson algebra via suitable relathionship which is given as follows 
\begin{pro}\label{Homp}
Let $(A,\succ,\prec,\alpha)$ be an admissible Hom-pre-Poisson algebra. Define
\begin{equation}\label{eq:sum}x\diamond y=x\succ y+x\prec y,\quad \forall x,y\in
  A.\end{equation}
Then $(\mathcal{P},\diamond,\alpha)$ is an admissible Hom-Poisson algebra, which is called the
\textup{\textbf{sub-adjacent admissible Hom-Poisson algebra}} of
$(\mathcal{P},\succ,\prec,\alpha)$ and denoted by $\mathcal{P}^c$ and $(\mathcal{P},\succ,\prec,\alpha)$ is
called  the {\bf compatible admissible Hom-pre-Poisson algebra} structure on
the admissible Hom-Poisson algebra $\mathcal{P}^c$.
\end{pro}\label{proposition1}    
\begin{proof}
 Let $(A,\succ,\prec,\alpha)$ be a admissible Hom-pre-Poisson algebra. For all
$x,y,z\in A$, we have
\begin{equation*}
\begin{aligned}
    &-(x\diamond y)\diamond \alpha(z) + \alpha(x)\diamond(y\diamond z) + \frac{1}{3} \Big( -\alpha(x)\diamond(z\diamond y) + \alpha(z)\diamond(x\diamond y) \\
    &+ \alpha(y)\diamond(x\diamond z) - \alpha(y)\diamond(z\diamond x) \Big) \\
    &= -\Big( (x\prec y + x\succ y)\prec \alpha(z) + (x\prec y + x\succ y)\succ \alpha(z) \Big) \\
    &+ \Big( \alpha(x)\prec(y\prec z + y\succ z) + \alpha(x)\succ(y\prec z + y\succ z) \Big) \\
    &+ \frac{1}{3} \Big( -\alpha(x)\prec(z\prec y + z\succ y) - \alpha(x)\succ(z\prec y + z\succ y) \\
    &+ \alpha(z)\prec(x\prec y + x\succ y) + \alpha(z)\succ(x\prec y + x\succ y) \Big) \\
    &= -(x\succ y)\succ \alpha(z) - (x\prec y)\succ \alpha(z) + \alpha(x)\succ(y\succ z) \\
    &+ \frac{1}{3} \Big( \alpha(x)\succ(z\prec y) - \alpha(z)\prec(x\succ y) - \alpha(z)\prec(x\prec y) \Big) \\
    &- \alpha(y)\succ(x\succ z) + \alpha(y)\succ(z\prec x) - \alpha(x)\succ(z\prec y) \\
    &+ (x\succ z)\prec \alpha(y) + \frac{1}{3} \Big( -\alpha(x)\succ(y\succ z) + \alpha(y)\succ(x\succ z) \\
    &+ \alpha(z)\prec(x\prec y) + \alpha(z)\prec(x\succ y) - \alpha(z)\prec(y\succ x) - \alpha(z)\prec(y\prec x) \Big) \\
    &- \alpha(z)\prec(x\succ y) - \alpha(z)\prec(x\prec y) + (z\prec x)\prec \alpha(y) \\
    &+ \frac{1}{3} \Big( -\alpha(z)\prec(y\succ x) - \alpha(z)\prec(y\prec x) + \alpha(y)\succ(z\prec x) \\
    &+ \alpha(x)\succ(z\prec y) - \alpha(x)\succ(y\succ z) \Big) \\
    &= A(x,y,z) + B(x,y,z) + C(x,y,z).
\end{aligned}
\end{equation*}
By Eqs.~(\ref{eq:presp1})-(\ref{eq:presp3}), we have
$A(x,y,z)=B(x,y,z)=C(x,y,z)=0$. Hence $(A^c,\diamond,\alpha)$ is an admissible Hom-pre-Poisson
algebra.   
\end{proof}
\begin{pro}
Let $(\mathcal{P},\succ,\prec,\alpha)$ be an admissible Hom-pre-Poisson algebra. Then
$(\mathcal{P},L_\succ,R_\prec,\alpha)$ is a representation of the sub-adjacent
 admissible Hom-Poisson algebra $(\mathcal{P}^c,\diamond,\alpha)$, where $L_\succ,R_\prec:\mathcal{P}\rightarrow \End_{\mathbb F}(\mathcal{P})$ are defined by
  \begin{equation}
   L_\succ(x)y=x\succ y,\quad R_\prec(x)y=y\prec x,\quad\forall x,y\in \mathcal{P}.
  \end{equation}
  Conversely, if $(\mathcal{P},\diamond,\alpha)$ is an admissible Hom-Poisson algebra together with two bilinear operations $\succ,\prec:\mathcal{P}\otimes \mathcal{P}\to \mathcal{P}$
such that $(L_\succ,R_\prec,\alpha)$ is a representation of $(\mathcal{P},\diamond,\alpha)$,
then $(\mathcal{P},\succ,\prec,\alpha)$ is an admissible Hom-pre-Poisson algebra.
\end{pro}
We get a direct consequence  given as follows.

\begin{cor} \label{cor:id}
Let $(\mathcal{P}, \prec, \succ,\alpha)$ be an admissible Hom-pre-Poisson algebra. Then the
identity map ${\rm id}$ is an $\mathcal O$-operator of the
sub-adjacent admissible Hom-Poisson algebra $(A^c,\diamond,\alpha)$ associated to the
representation $(\mathcal{P},L_\succ,R_\prec,\alpha)$.
\end{cor}
In the following, we recall the notion of Hom-pre-Poisson algebras (see \cite{makhlouf}), which is a Hom-Zinbiel algebra and Hom-pre-Lie algebra satistfying a given compatibility. We show that any admissible Hom-pre-Poisson algebra gives a Hom-pre-Poisson algebra. 
\begin{defi}
 A \textbf{Hom-pre-Lie} algebra is a triple
 $(A,\star,\alpha)$ consisting of a vector space $A$, a bilinear map  $\star:A\otimes A\longrightarrow A$ and an algebra morphism $\alpha:A\longrightarrow A$ satisfying
 \begin{align}
 &(x\star y)\star\alpha(z)-\alpha(x)\star(y\star z)=(y\star x)\star\alpha(z)-\alpha(y)\star(x\star z)
 \end{align}
\end{defi}
\begin{defi}
  A \textbf{Hom-Zinbiel} algebra is a triple $(A,\diamond,\alpha)$ consisting of a vector space  $A$, a bilinear map $\diamond:A\otimes A\longrightarrow A$
 and an algebra morphism $\alpha:A\longrightarrow A$ satisfying 
 \begin{align}
 \alpha(x)\diamond(y\diamond z)=(x\diamond y)\diamond \alpha(z)+(y\diamond x)\diamond\alpha(z), \forall x,y,z \in A.
 \end{align}
\end{defi}
\begin{defi}\cite{makhlouf}
A \textbf{Hom-pre-Poisson} algebra is a quadruple $(A,\diamond,\star,\alpha)$ where $(A,\cdot,\alpha)$ is a Hom-Zinbiel algebra and $(A,\star,\alpha)$ is a Hom-pre-Lie algebra such that $\forall x,y,z\in A$, the following conditions holds:
\begin{align}
& (x\star y-y\star x)\cdot \alpha(z)=  \alpha(x)\star(y\cdot z)-\alpha(y)\cdot(x\star z)\\
&(x\cdot y+y\cdot x)\star\alpha(z)=\alpha(x)\cdot(y\star z)+\alpha(y)\cdot (x\star z).
\end{align}
A Hom-pre-Lie algebra $(A,\cdot,\star,\alpha)$ is called regular if $\alpha$ is invertible.
\end{defi}
\begin{pro}\label{HomPoi}
Let $(A,\cdot,\star,\alpha)$ be a Hom-pre-Poisson. Define $$x\circ y=x\cdot y+y\cdot x, \quad \quad [x,y]=x\star y-y\star x,\forall x,y\in A.$$ Then $(A,[\cdot,\cdot],\circ,\alpha)$ is a Hom-Poisson algebra.    
\end{pro}
By direct computation we get the following theorem
\begin{thm}\label{Hompd}
 Any admissible Hom-pre-Poisson algebra $(P,\prec, \succ,\alpha)$ gives  a Hom-pre-Poisson algebra $(P,\cdot, \diamond,\alpha)$, where
 $$x\cdot y= \frac{1}{2}(x\prec y-y \succ x)\quad \text{and}\quad x\diamond y= \frac{1}{2}(x\prec y+y \succ x),\quad \forall x,y\in P.$$
\end{thm}
\begin{defi}
 Let $(A,\diamond,\alpha)$ be an admissible Hom-Poisson algebra and $(V\frkl,\frkr,\beta)$ be a
representation of $(A,\diamond,\alpha)$. A linear map $T:V\to A$ is called
an \textbf{\textup{$\mathcal{O}$-operator of $(A,\diamond,\alpha)$ associated
to}} $(V\frkl,\frkr,\beta)$ if $T$ satisfies
    \begin{align}
  &  T\circ\beta=\alpha\circ T\\
&T(u)\diamond T(v)=T(\frkl(T(\beta^{-1}(u))v+\frkr(T(\beta^{-1}(v))u)),\quad \forall u,v\in V.
    \end{align}   
\end{defi}

\begin{ex}
Let $(A,\diamond,\alpha)$ be an admissible Hom-Poisson algebra. An $\mathcal O$-operator
$\mathcal{R}$ associated to the representation $(A,L,R,\alpha)$ is called a {\bf
Hom-Rota-Baxter operator of weight zero}, that is, $R$ satisfies
\begin{align}
&\mathcal{R}\circ \alpha=\alpha\circ\mathcal{R}\\
&\mathcal{R}(x)\diamond \mathcal{R}(y)=\mathcal{R}(\mathcal{R}(\alpha^{-1}(x))\diamond y+x\diamond \mathcal{R}(\alpha^{-1}(y))),\;\;\forall x,y\in A.
\end{align}
\end{ex}

\begin{lem}
Let $(\mathcal{P},\diamond,\alpha)$ be an admissible Hom-Poisson algebra and $(V,\rho,\mu,\beta)$ a representation of $\mathcal{P}$ on $V$ and $T:V\to A $ be  a Hom-$\mathcal O$-operator. Define $\prec_{T}$   and $\succ_{T}$ on $V$ by 
\begin{align*}
 &x\prec_{T} y=\mu(T(\beta^{-1}(x))y \\
 &x\succ_{T}y=\rho((T(\beta^{-1}(x)))y
\end{align*}
Then $(\cal{P},\prec_{T},\succ_{T},\alpha)$ is an admissible Hom-pre-Poisson algebra.
\end{lem}

We have shown that (Hom-)Poisson, admissible (Hom-)Poisson, Hom-pre-Poisson, admissible Hom-pre-Poisson algebras are
closely related in the sense of commutative diagram of categories as follows:
\begin{equation}\label{diagramhommalcev}
    \begin{split}
 {\xymatrix{
\ar[rr] \mbox{\bf Adm. Hom-pre-Poi. alg  }\ar[d]_{\mbox{Prop. \ref{Homp}}}\ar[rr]^{\mbox{ Theorem \ref{Hompd}}}_{\mbox{ \ }}
                && \mbox{\bf  Hom-pre-Poisson alg  }\ar[d]_{\mbox{Prop. \ref{HomPoi}}}\\
\ar[rr] \mbox{\bf Adm. Hom-Poisson alg }\ar@<-1ex>[u]_{\mbox{R-B }}\ar[rr]^{\mbox{\quad\quad Prop. \ref{Yau}\quad\quad  }}
                && \mbox{\bf Hom-Poisson alg  }\ar@<-1ex>[u]_{\mbox{R-B}}\\
          \ar[rr] \mbox{\bf Adm. Poisson alg. }\ar@<-1ex>[u]_{ \mbox{Twist}}\ar[rr]^{\mbox{Commutator}}_{\mbox{Anti-Commutator}}
                && \mbox{\bf Poisson alg  }\ar@<-1ex>[u]_{\mbox{Twist}}}
}
 \end{split}
\end{equation}
\section*{Acknowledgement} The author would like to thank the referee for valuable comments and suggestions on this article.

\end{document}